\documentclass[twocolumn]{IEEEtran} 

\usepackage{epsfig}
\ifCLASSINFOpdf
\else
\fi
\usepackage[cmex10]{amsmath}

\graphicspath{{fig/}}

\usepackage{csquotes}  
\newcommand\q{\enquote}

\usepackage{amsfonts}
\usepackage{amssymb}

\usepackage{tikz}
\usetikzlibrary{matrix,positioning,decorations.pathreplacing}

\usepackage{ifthen}

\newtheorem{nnassumption}{\bf Assumption}
\newenvironment{assumption}{\begin{nnassumption}\it}{\end{nnassumption}}
\newtheorem{nntheorem}{\bf Theorem}
\newenvironment{theorem}{\begin{nntheorem}\it}{\end{nntheorem}}
\newtheorem{nncorollary}{\bf Corollary}
\newenvironment{corollary}{\begin{nncorollary}\it}{\end{nncorollary}}
\newtheorem{nndefinition}{\bf Definition}
\newenvironment{definition}{\begin{nndefinition}\it}{\end{nndefinition}}
\newtheorem{nnproposition}{\bf Proposition}
\newenvironment{proposition}{\begin{nnproposition}\it}{\end{nnproposition}}
\newtheorem{nnproblem}{\bf Problem}

\newtheorem{nnlemma}{\bf Lemma}
\newenvironment{lemma}{\begin{nnlemma}\it}{\end{nnlemma}}
\newtheorem{nnremark}{\bf Remark}
\newenvironment{remark}{\begin{nnremark} \rm }{\hfill \hspace*{1pt}\hfill $\circ$\end{nnremark}}
\newtheorem{nnexample}{\bf Example}

\newenvironment{proof}{{\bf Proof.}}{\hfill \hspace*{1pt}\hfill $\Box$}

\newcommand\sat{\textup{sat}}

\newcommand{\lel}{\left\langle}
\newcommand{\rir}{\right\rangle}

\newcommand{\scalp}[2]{ \lel #1, #2 \rir }

\newcommand \KL  {\mathcal{KL}}

\newcommand{\intt}{{\rm int}\,}

\usepackage{color}

\def\R{\mathrm{I\kern-0.21emR}}
\def\N{\mathrm{I\kern-0.21emN}}

\renewcommand{\geq}{\geqslant}
\renewcommand{\leq}{\leqslant}

\begin{document}
\title{Local stabilization of an unstable parabolic equation via saturated controls}

\author{Andrii Mironchenko, Christophe Prieur and Fabian Wirth

\thanks{A. Mironchenko is with 
Faculty of Computer Science and Mathematics, University of Passau,
94030 Passau, Germany.
Email: andrii.mironchenko@uni-passau.de. Corresponding author.
}
\thanks{
C. Prieur is with Univ. Grenoble Alpes, CNRS, Grenoble INP, GIPSA-lab, 38000 Grenoble, France. Email: Christophe.Prieur@gipsa-lab.fr
}
\thanks{F. Wirth is with 
Faculty of Computer Science and Mathematics, University of Passau,
94030 Passau, Germany.
Email: fabian.(lastname)@uni-passau.de.
}
\thanks{
A. Mironchenko has been supported by DFG, grant Nr. MI 1886/2-1.
C. Prieur has been partially supported by MIAI @ Grenoble Alpes, (ANR-19-P3IA-0003). 
A preliminary version of this paper was presented at 
the 11-th IFAC Symposium on Nonlinear Control Systems (NOLCOS 2019), see \cite{MPW19a}.
}
}

\maketitle


\begin{abstract}
We derive a saturated feedback control, which locally stabilizes a linear reaction-diffusion equation.
In contrast to most other works on this topic, we do not assume the Lyapunov stability of the uncontrolled system and consider general unstable systems.
Using Lyapunov methods, we provide estimates for the region of attraction for the closed-loop system, given in terms of linear and bilinear matrix inequalities.
We show that our results can be used with distributed as well as scalar boundary control, and with different types of saturations.
The efficiency of the proposed method is demonstrated by means of numerical simulations.
\end{abstract}



\section{Introduction}
\label{sec:Intro}


In applications of control technology, physical inputs (like force,
torque, thrust, stroke, etc.)  are often limited in size \cite{BeM95}. 
If such input limitations are neglected, this may result in undesirable
oscillations of the closed-loop system, lack of global stabilizability and
in a dramatic reduction of the region of attraction of the closed-loop system, see e.g. \cite{TGS11,zaccarian2011modern,SSS12}
for an introduction to the nonlinear behavior induced by input limitations.

This leads to the problem of (local or global) stabilization of control systems with inputs of a norm not exceeding a prescribed value.

In this paper, we study the stabilizability of a one-dimensional linear unstable reaction-diffusion equation (also called heat equation) using a saturated control. Many results exist in the literature for the control of this class of equations, using either bounded or unbounded control operators, with or without input delays. 
More specifically, in \cite{Krstic:SCL:2009} a backstepping approach is
applied to design a boundary delayed feedback control for a heat equation
(see \cite{KrsticSmyshlyaev:book:08} for further results using the same
design methods). In \cite{FridmanOrlov2009} a stable heat
partial differential equation (PDE) is controlled by means of a delayed
bounded linear control operator (see also \cite{solomon2015stability} for
the semilinear case). 
For the case of unbounded control operators, see \cite{nicaise:dcds:2009,PrT19} for the computation of delayed control to stabilize the reaction-diffusion equation. 

To analyze the effect of input saturations and to design saturated
controllers for infinite-dimensional systems, different results and
techniques are available. One of the first papers in this field is
\cite{Sle89}, where compact and bounded control operators are
considered. Another notable early reference is \cite{LaS03} where an
observability assumption is stated for the study of PDEs with constrained
controllers. 
In these results, the control inputs are functions from a certain input space, and the saturation map has to be
understood as a limitation on the norm of the input in this space. 
A more physically relevant type of saturation is given by pointwise saturations, which limit the values of the control input at each point by a prescribed value.

%
Pointwise saturations are more complex and require further developments, some of which we investigate in the present paper. Other types of saturation functions can  also be useful in practice, see \cite{MCP18} for a discussion. 

The definition of the saturation map used in the present paper is aligned
to that used in
\cite{PTS16} where Lyapunov methods are shown to be useful for the
stability analysis of wave equations subject to saturated inputs. See also
\cite{MCP18, MAP17}, where systems in Hilbert spaces with applications to the Korteweg--de Vries equation have been addressed. 



\emph{Our main result is the derivation of a saturated feedback control,
  which locally stabilizes the unstable heat equation, and for which
  estimates of the region of attraction for the closed-loop system can be
  provided. These estimates are formulated in terms of matrix inequalities which can be efficiently solved numerically.}
\emph{
We emphasize that, in contrast to the works \cite{Sle89,
    LaS03,PTS16}, where the stress is on the global stabilization of marginally stable hyperbolic systems (which have infinitely many eigenvalues on the imaginary axis), in our paper we consider parabolic systems, which have finitely many exponentially unstable modes. Therefore
  the control of the system with a saturating controller only provides local asymptotic stability, paralleling what is known for finite-dimensional systems (see in particular \cite{TeelNestedSaturation}).} 

We would like to mention the related research on model predictive control of parabolic PDEs under constraints on the state and the input  \cite{DEM06}, and on Lyapunov-based control of parabolic systems by controls of a bounded magnitude \cite{EAC03}.
Note however, that in these papers the problem of estimating the region of attraction has not been studied, which is the key objective of this work.
In \cite{KaF17} stabilizing backstepping-based
boundary controllers for coupled heat-ODE systems with time-varying state delays in the presence of actuator saturation are developed, and in \cite{KaF18} boundary stabilization of a nonlinear Schr\"odinger equation with state delay and bounded internal disturbance
is performed. Both in \cite{KaF17, KaF18} estimates of the region of attraction of the closed-loop system are provided.

Our approach is based on the spectral decomposition of the open-loop dynamics of a heat equation into the finite-dimensional unstable part and infinite-dimensional stable dynamics, which is a classical tool used in PDE control \cite{Rus78, BaT04, BLT06}.
To stabilize the unstable finite-dimensional part using a saturated control we apply techniques, which are well-known for ODE systems (see e.g.  \cite{TGS11,zaccarian2011modern,SSS12} for systems with non-delayed input, and \cite{LiF14} for systems with a delayed input).

Using Lyapunov functions, we derive in Section~\ref{sec:Estim_attr_region_infdim_sys} linear matrix inequalities
(LMIs) whose solutions provide estimates of the region of attraction
of the closed-loop finite-dimensional system, see \cite{BEF94,
  ScW00, VaB00} for an introduction to matrix inequalities. 
We note, that a different Lyapunov function has been used for local ISS stabilization of the diffusion equation by saturated controls in the paper \cite{TMP18}.
Then we show how asymptotic stability  and estimates for the region of attraction can
be obtained for the reaction-diffusion equation in closed-loop with the nonlinear saturated control. 

We show that our results can be extended in a variety of different directions.
In Section~\ref{sec:anti-windup} we demonstrate, how dynamic controllers can be used to enlarge the region of attraction. In Section~\ref{sec:Estim_attr_region_infdim_sys_b_times_u}
we show that more complex types of
saturation functions can be tackled, such as pointwise ($L_\infty$) saturation functions. 
\textit{Although our main results concern the stabilization via
  distributed controllers, in Section \ref{sec:Boundary_control} we show
  how the methods can be applied to unbounded scalar control operators}, designed as the output of a finite-dimensional boundary control plant, which is fed with the saturated input.


In this light, our approach seems to be useful for any infinite-dimensional
systems for which there exist only finitely many unstable modes and which
is to
be controlled through a bounded input operator (as is the case for systems with input delays considered in e.g. \cite{FlP86}), see also Remark~\ref{rem:Stabilization-with-a-prescribed-decay-rate}.

The remaining part of this paper is organized as follows. We first
introduce the reaction-diffusion equation with bounded input operator in
Section \ref{sec:Start}. The saturation function is defined and the spectral decomposition is provided. 
A saturated feedback is designed in Section
\ref{sec:Estim_attr_region_infdim_sys} and an estimate for the region of attraction is provided, first for the open-loop unstable part, and then for both stable and unstable part.  
Numerical experiments conducted in Section~\ref{sec:Numerics} illustrate our design method and the obtained estimates of the region of attraction depending on the saturation level.  
Other saturations are considered in Section
\ref{sec:Estim_attr_region_infdim_sys_b_times_u}. In Section
\ref{sec:Boundary_control}, we show how our results can be applied in the
case of
unbounded input operators, that is to saturating boundary controllers
resulting of a finite-dimensional dynamical system. Concluding remarks and
possible future lines of research are collected in Section \ref{sec:conclusion}.

%

\textbf{Notation:} The set of nonnegative reals we denote by $\R_+$. The Euclidean norm on $\R^n$ is denoted by $|\cdot|$,
the operator norm induced by this norm on spaces of matrices is denoted by
$\|\cdot\|$. The interior of a set $S$ in a topological space is denoted
$\intt S$ 
and $\overline{S}$ denotes its closure.
In any normed vector space, the ball of radius $r$ around $0$ is denoted by $B_r(0)$.
By $\N$ we denote the nonnegative integers. For convenience,
$\N^*:=\N\setminus\{0\}$.  For $k \in \N, L>0$, $H^{k}(0,L)$ denotes the
Sobolev space of functions from the space $L_2(0,L)$, which have weak
derivatives of order $\leq k$, all of which belong to
$L_2(0,L)$. $H^{k}_0(0,L)$ is the closure of $C_0^k(0,L)$ 
(the $k$-times
continuously differentiable functions with compact support in $(0,L)$)
 in
the norm of $H^{k}(0,L)$.

\section{Problem formulation}
\label{sec:Start}

We consider the stabilization problem of a one-dimensional linear reaction-diffusion equation by means of a
distributed control $u:\R_+\to\R^m$. Let $L>0$. We are given $m$ functions
$b_k: [0,L] \to \R$, $k=1,\ldots,m$ which describe at which places the control input $u_k\in\R$
is acting. The function $c$ models the place-dependent reaction rate. The
system model is then  
\begin{equation}\label{closed-loop:sat}
\begin{split}
& w_t(t,x)=w_{xx}(t,x)+c(x)w(t,x)\\
&\qquad\qquad\quad + \sum_{k=1}^m b_k(x) \sat(u_k(t)),  \ t>0, \; x \in (0,L), \\
& w(t,0)=w(t,L)=0, \quad t>0, \\
& w(0,x)=w^0(x) , \quad x \in (0,L).
\end{split}
\end{equation}
We assume that the state space of this system is $X:=L_2(0,L)$ and that
$c,b_k \in X$, $k=1,\ldots,m$.

Here $\sat$ is a component-wise saturation function, that is, for all $k=1,\ldots, m$ and for any $v \in\R^m$,
\begin{equation}
\label{sat:finite:dim}
\sat(v)_k
:=
\begin{cases}
v_k & \text{ if }  |v_k|\leq \ell, \\ 
\frac{\ell}{|v_k|}v_k & \text{ if } | v_k |\geq \ell,
\end{cases}
\end{equation}
where $\ell>0$ is the given level of the saturation, which is assumed to be uniform with respect to the index $k$.

\begin{remark}
\label{rem:Saturations_in_boundary_control} 
Systems of the form \eqref{closed-loop:sat} also occur in the problem of stabilizing the linear heat equation by means of boundary control subject to delays or saturation, see e.g. \cite{PrT19} as well as Section~\ref{sec:Boundary_control} below.
\end{remark}

For the design of a stabilizing feedback, we use the well-known eigenfunction decomposition method.
Define 
\begin{equation}
A=\partial_{xx}+c(\cdot)\mathrm{id}: X \to X
\end{equation}
with domain $D(A)=H^2(0,L)\cap H^1_0(0,L)$. 
Then the control system \eqref{closed-loop:sat} takes the form
\begin{equation}
\label{newzero}
w_t(t,\cdot)=Aw(t,\cdot)+ \sum_{k=1}^m  b_k \sat(u_k(t)).
\end{equation}

We note that $A$ is selfadjoint and has compact resolvent, see Appendix~\ref{appendix}. 
Hence, the spectrum of $A$ consists of only isolated eigenvalues with finite multiplicity, 
see \cite[Theorem~III.6.29]{Kat95}. 
Furthermore, there exists a Hilbert basis $(e_j)_{j\geq 1}$ of $X$ consisting of eigenfunctions of $A$, associated with the sequence of  eigenvalues $(\lambda_j)_{j\geq 1}$. Note that
\begin{equation*}
-\infty<\cdots<\lambda_j<\cdots<\lambda_1\quad \textrm{and}\quad \lambda_j\underset{j\rightarrow +\infty}{\longrightarrow}-\infty     
\end{equation*}
and that $e_j(\cdot)\in D(A)$ for every $j\geq 1$.

We consider (mild) solutions of the system \eqref{closed-loop:sat} (see \cite[Section 3.1]{CuZ95}), which exist and are unique for any initial condition in $X$ and for any $u_k \in L_{1,loc}([0,\infty))$, for $k=1,\ldots,m$.

Every solution $w(t,\cdot)\in D(A)$ of \eqref{newzero} can be
expanded as a series in the eigenfunctions $e_j(\cdot)$, convergent in
$H_0^1(0,L)$,
\begin{equation}
    \begin{aligned}
w(t,\cdot)&=\sum_{j=1}^{\infty}w_j(t)e_j(\cdot),\\
   w_j(t)&:=\langle w(t,\cdot),e_j(\cdot)\rangle_{L_2(0,L)},\ j\in\N^*.
    \end{aligned}
\label{eq:Orthogonal_decomposition_w}
\end{equation}
Analogously, we can expand the coefficients $b_k$ in the series 
\begin{equation*}
b_k(\cdot) = \sum_{j=1}^{\infty}b_{jk} e_j(\cdot), \quad  b_{jk}=\langle b_k(\cdot),e_j(\cdot)\rangle_{L_2(0,L)},\ j \in \N^*.
\label{eq:Orthogonal_decomposition_b}
\end{equation*}
As discussed in Appendix~\ref{appendix2}, \eqref{newzero} is
equivalent to the infinite-dimensional control system
\begin{align}
\label{sys-dim-infinie}
\dot{w}_j(t) &= \lambda_jw_j(t)+ \sum_{k=1}^m  b_{jk}\sat(u_k(t)) \nonumber\\
&=   \lambda_jw_j(t)+  \mathbf{b}_{j} \cdot \sat(u(t)) ,\qquad
j\in\N^*,
\end{align}
where \q{$\cdot$} is the scalar product in $\R^m$, $\sat(u(t)) \in \R^m$ is the vector with entries $\sat(u_k(t))$
and $\mathbf{b}_{j}$ is the row vector with entries $b_{jk}$, $k=1,\ldots,m$.

Let $n\in\N^*$ be the number of nonnegative eigenvalues of $A$
and let $\eta>0$ be such that
\begin{equation}\label{refeta}
\forall j>n:\quad \lambda_j<-\eta<0.
\end{equation}

With the matrix notations
\begin{equation}
\label{eq:def:A1}
z{:=}\hspace{-1mm}\begin{pmatrix} w_1 \\ \vdots \\ w_n
\end{pmatrix}\! ,
\mathbf{A}{:=}\hspace{-1mm}\begin{pmatrix}
 \lambda_1 &\! \cdots \! &    0           \\
 \vdots         & \!\ddots \! &   \vdots       \\
  0              &\! \cdots \!& \lambda_n
\end{pmatrix}\! , 
\mathbf{B}{:=}\hspace{-1mm}\begin{pmatrix}  b_{11} &\!\! \cdots \!\! & b_{1m} \\ \vdots &&
    \vdots\\ b_{n1} &\!\! \cdots\!\! & b_{nm}  \end{pmatrix} 
\end{equation}
the $n$ first equations of 
\eqref{sys-dim-infinie}
form the unstable finite-dimensional control system 
\begin{equation}\label{systfini}
\dot{z}(t)=\mathbf{A} z(t) + \mathbf{B}\sat(u(t)).
\end{equation}

\begin{remark}
\label{rem:Stabilization-with-a-prescribed-decay-rate} 
 \textbf{(Possible extensions of our results).}

(i) The same method can be used for stabilization with a prescribed
decay rate $-\beta$. Just define for a suitable $\eta >
\beta$ the minimal $n$ such that \eqref{refeta} holds, and then use the
method described here. 

(ii) Since the input operator is bounded and the input space is
finite-dimensional, the system can only be stabilized if the unstable
spectrum consists of finitely many eigenvalues (counting multiplicities),
see \cite[Theorem 1.1]{JaZ99} and the discussion in that paper relating to the
history of the result. If this restriction is satisfied (e.g., this is the case for systems with input delays considered in \cite{FlP86}), the methods of this paper apply.
\end{remark}

\section{Estimation of the region of attraction for saturated inputs}
\label{sec:Estim_attr_region_infdim_sys}

\subsection{Decomposition of the system into stable and unstable part}

We now introduce a decomposition of the state space into a finite-dimensional space on which the stabilization problem has to be solved and
its orthogonal complement, which is invariant under the free dynamics.

Let $X_n$ be the subspace of $L_2(0,L)$ spanned by
$(e_i(\cdot))_{i=1}^n$ and $\pi_n$ be the orthogonal projection onto $X_n$, that is
\begin{equation}
\label{newun}
\pi_n w(t,\cdot):=\sum_{j=1}^nw_j(t)e_j(\cdot).
\end{equation}
We define also $X_n^{\bot}$ as the orthogonal complement
of $X_n$ in $X$.
Let $\iota: \R^n \rightarrow X_n$ be the isomorphism defined by
$\iota(e^j)=e_j(\cdot)$, where $(e^j)_{j=1,\ldots, n}$ is the canonical
basis of $\R^n$. 
We will use the isometric representation of
$L_2(0,L)$ as $\ell_2(\N^*, \R)$ obtained by the isomorphism
induced by $e_j(\cdot)\mapsto e^j$, where
\[
\ell_2(\N^*, \R):=\Big\{(x_k)_{k\in\N^*}\in\R^{\N^* }\ : \ \sum_{k=1}^\infty|x_k|^2 <\infty\Big\},
\]
where $e^j, j\in \N^*$ are the standard basis vectors in $\ell_2(\N^*,
\R)$ and we use the standard norm on that space. 
Corresponding to the decomposition
$L_2(0,L) = X_n \bigoplus X_n^{\bot}$, where \q{$\bigoplus$} is the orthogonal sum of subspaces, we denote 
$\ell_2(\N^*, \R) = \R^n \bigoplus \ell_{2,j>n}$, where we identify $\R^n$ with the sequences with
support in $\{1, \ldots,n \}$ and $\ell_{2,j>n}$ is the set of sequences
in $\ell_2(\N^*, \R)$ which are $0$ in the first $n$ entries.

Given a linear map $K:X_n \to \R^m$, consider the feedback 
\begin{align}
u=&K \pi_n{w}(\cdot) = K \Big( \sum_{j=1}^nw_je_j(\cdot) \Big) 
=  \sum_{j=1}^n w_jKe_j(\cdot) \nonumber\\
=&  \sum_{j=1}^n w_j\mathbf{K}_j =  \mathbf{K} z,
\label{eq:Stabilizing_controller}
\end{align}
where $\mathbf{K}_j:=Ke_j(\cdot) \in \R^m$, $j=1,\ldots,n$, and where we
use 
the notation from (\ref{eq:def:A1}) in the final step and set 
\begin{center}
$\mathbf{K}:=(\mathbf{K}_1,\ldots, \mathbf{K}_n)\in \R^{m\times n}$.
\end{center}

Hence the system \eqref{sys-dim-infinie} with the feedback \eqref{eq:Stabilizing_controller} is equivalent to the following set of differential equations:
\begin{eqnarray}
\dot{w}_j(t)= \lambda_j w_j(t) + \mathbf {b}_{j}\cdot
\sat(\mathbf{K}z(t)),\quad j \in \N^*.
\label{eq:Componentwise_equations_with_saturation_1_FW}
\end{eqnarray}
Using \eqref{eq:def:A1}, we rewrite the first $n$ equations of \eqref{eq:Componentwise_equations_with_saturation_1_FW} as
\begin{eqnarray}
\dot{z}(t) = \mathbf{A}z(t) + \mathbf{B} \sat(\mathbf{K}z(t)).
\label{eq:Systfini_saturated_FW}
\end{eqnarray}
Now, \eqref{eq:Componentwise_equations_with_saturation_1_FW} can be considered as a cascade interconnection of an $n$-dimensional part, described by the equations
\eqref{eq:Systfini_saturated_FW} and of an infinite-dimensional part described by the equations
\begin{eqnarray}
\dot{w}_j(t)&=& \lambda_j w_j(t) + \mathbf {b}_{j} \cdot \sat(\mathbf{K}z(t)),\quad j \geq n+1.
\label{eq:Componentwise_equations_with_saturation_infdimpart}
\end{eqnarray}

Our general assumption is:
\begin{assumption}
\label{ass:Stabilizability} 
The pair $(\mathbf{A},\mathbf{B})$ is stabilizable, i.e. there is $\mathbf{K}\in \R^{m\times n}$:
$\mathbf{A}+\mathbf{B}\mathbf{K}$ is Hurwitz.
\end{assumption}

\begin{remark}
\label{rem:Saturated_Stabilizability_LinearStabilizability} 
Clearly, the closed-loop system \eqref{eq:Systfini_saturated_FW} with a matrix $\mathbf{K}$ as in Assumption~\ref{ass:Stabilizability} 
is locally asymptotically stable. 

Conversely, if a feedback $\mathbf{K}$ renders the closed-loop system \eqref{eq:Systfini_saturated_FW}
locally asymptotically stable, then also
\begin{eqnarray}
\dot z= \mathbf{A}z + \mathbf{B} u
\label{eq:Linear_fin_dim_sys_no_sat}
\end{eqnarray}
is locally and hence globally asymptotically stabilized by means of the feedback $u(t):=\mathbf{K} z(t)$. Thus, local asymptotic stability of 
\eqref{eq:Systfini_saturated_FW} implies that the pair $(\mathbf{A},\mathbf{B})$ is stabilizable.

We note that in the case $m=1$ the situation simplifies further 
as then \eqref{eq:Linear_fin_dim_sys_no_sat} is a linear diagonal system
with scalar control input. The criterion for stabilizability is then that
$b_{j1} \neq 0$ for all $j=1,\ldots,n$ and $\lambda_k \neq \lambda_j$ for all $k,j =1,\ldots,n$, $k\neq j$ (which is an easy exercise).
That is, the localization function
$b_1$ should not be orthogonal to an unstable eigenfunction and that all
unstable eigenvalues need to be simple.
\end{remark}

Next we show that the problem of exponential stabilization of the overall system \eqref{sys-dim-infinie}
boils down to the exponential stabilization of the finite-dimensional unstable system \eqref{systfini}. 
This latter problem will be elaborated in Section~\ref{sec:Finite-dim_Saturated_inputs}.

\begin{definition}
\label{def:Region_of_Attraction} 
Assume that $\mathbf{K}$ is chosen so that $0$ is a locally asymptotically
stable fixed point of \eqref{eq:Systfini_saturated_FW}.  We say that $S$
is a \emph{region of attraction} of $0$ if
\begin{itemize}
	\item[(i)]  $0 \in \intt S$; 
	\item[(ii)] for any
$z_0 \in S$ the corresponding solution of \eqref{eq:Systfini_saturated_FW}
satisfies $z(t;z_0) \to 0$ as $t\to\infty$;
	\item[(iii)] $S$ is forward invariant, i.e. for any $z_0\in S$ it holds that
$z(t;z_0) \in S$ for all $t\geq0$.
\end{itemize}
The largest set (with respect to set inclusion) with the properties (i)-(iii) is
called the \emph{maximal region of attraction}.
\end{definition}

As unions of regions of attraction are again a region of attraction, it is
immediate that the maximal region of attraction is uniquely defined in
this way and coincides with what is called domain of attraction in \cite{Hah67}.

\begin{definition}
\label{def:LocExpStab_with_DomAttr_S} 
\eqref{eq:Systfini_saturated_FW} is called \emph{locally exponentially stable in 0 with region of attraction $S$}, if the following conditions hold:
\begin{itemize}
	\item[(i)] there exist $\varepsilon, M, a >0$ such that for any initial condition satisfying $|z_0|\leq \varepsilon$, it holds
\begin{eqnarray}
|z(t;z_0)|\leq Me^{-at} |z_0| \quad \forall t\geq 0
\label{eq:exp-decay}
\end{eqnarray}
	\item[(ii)]
$\overline{B_\varepsilon(0)} \subset S$ and $S$ is a region of attraction of \eqref{eq:Systfini_saturated_FW}.
\end{itemize}
We call \eqref{eq:Systfini_saturated_FW} \emph{globally exponentially stable in 0} if  \eqref{eq:exp-decay} holds for some $a,M>0$ and all $z_0 \in\R^n$.
\end{definition}
Definitions~\ref{def:Region_of_Attraction} and \ref{def:LocExpStab_with_DomAttr_S} can be stated analogously for system \eqref{eq:Componentwise_equations_with_saturation_1_FW}.

We note that if the maximal region of attraction is not $\R^n$, then the system
cannot be exponentially stable on the maximal region of attraction (for
Lipschitz continuous systems), see \cite{Hah67}. Thus by analyzing regions of
attraction with exponential stability we necessarily restrict the region
of attraction.

%
\begin{proposition}
\label{prop:Attraction_region_finite_dim_and_infinite_dim} 
Assume $\mathbf{K}$ is chosen such that the subsystem \eqref{eq:Systfini_saturated_FW} is locally
exponentially stable in $0$ with region of attraction $S \subset \R^n$.
Then:
\begin{itemize}
	\item[(i)] system
          \eqref{eq:Componentwise_equations_with_saturation_1_FW} is
          locally exponentially stable in $0$ with region of attraction
$S \times \ell_{2,j>n}$.
	\item[(ii)] system \eqref{closed-loop:sat} with the feedback
          \eqref{eq:Stabilizing_controller} is locally exponentially
          stable in $0$ with region of attraction $\iota(S) \times X_n^{\bot}$.
\end{itemize}

  In addition, for any closed and bounded set $G\subset \intt(\iota(S) \times X_n^{\bot})$, there exist two positive values $M$ and $a$ such that for any initial condition $w(0,\cdot)$ in $G$, the solution $w(\cdot)$ to 
	\eqref{closed-loop:sat} with the controller \eqref{eq:Stabilizing_controller} satisfies
\begin{equation}\label{4:juin}
\|w(t,\cdot)\|_X \leq Me^{-at} \|w(0,\cdot)\|_X \quad \forall t\geq 0.
\end{equation} 
\end{proposition}

\begin{proof}
Pick a compact subset $G'$ of $\intt S$. Since we assume that
\eqref{eq:Systfini_saturated_FW} is locally exponentially stable in 0 with
region of attraction $S \subset \R^n$, it follows from a standard
compactness argument that there exist $M,a>0$ so that for any $z_0\in G'$ the solution $z(\cdot;z_0)$ to \eqref{eq:Systfini_saturated_FW} satisfies
\[
|z(t;z_0)|\leq Me^{-at} |z_0|, \quad t\geq 0.
\]

From equations \eqref{sys-dim-infinie} and
\eqref{eq:Stabilizing_controller}, we derive that for $j=n+1,\ldots,
\infty$, for any $t\geq 0$ and for any $(w_{n+1}(0),w_{n+2}(0),\ldots) \in
\ell_{2,j>n}$ it holds that
\[
w_j(t) = e^{\lambda_j t} w_j(0) +  \mathbf {b}_{j} \cdot \int_0^t e^{\lambda_j (t-s)} \sat(\mathbf{K}z(s)) ds.
\]
From \eqref{sat:finite:dim} it follows that for all $z \in \R^n$ we have
\[
|\sat(\mathbf{K}z)| \leq |\mathbf{K}z| \leq \|\mathbf{K}\| |z|.
\]
Also due to the Cauchy-Bunyakovsky-Schwarz inequality we have that, for all $j\in\N$,
\begin{eqnarray}
|b_{jk}|=\Big|\langle b_k(\cdot),e_j(\cdot)\rangle_{X}\Big| \leq \|b_k\|_X \|e_j\|_X = \|b_k\|_X.
\label{eq:b_j_estimate}
\end{eqnarray}
Thus, for all $j\geq n+1$,
we obtain exploiting \eqref{eq:b_j_estimate} that
\begin{align*}
|w_j(t)| 
				 \leq&  e^{-\eta t} |w_j(0)| + |\mathbf{b}_j| \int_0^t e^{ - \eta (t-s)} |\mathbf{K}z(s)| ds\\
				 \leq&  e^{-\eta t} |w_j(0)| + |\mathbf{b}_j|\|\mathbf{K}\| \int_0^t e^{ - \eta (t-s)} Me^{-as} |z(0)| ds\\
				 =&  e^{-\eta t} |w_j(0)| +  |\mathbf{b}_j| \frac{M\|\mathbf{K}\|}{\eta - a} (e^{-at} - e^{-\eta t}) |z(0)|.
\end{align*}
The above computations have been performed for the case when $\eta \neq a$. If $a= \eta$, then
it holds that
\begin{eqnarray*}
|w_j(t)| &\leq&  e^{-\eta t} |w_j(0)| + M  |\mathbf{b}_j|\|\mathbf{K}\| t e^{-\eta t}|z(0)|.
\end{eqnarray*}

Now, using the inequality $(a+b)^2\leq 2(a^2 + b^2)$ for any $(a,b)\in
\R^2$, and the square summability of $|w_j(0)|$ and $|b_{jk}|$, $k=1,\ldots,m$, it follows that $\sum_{j=n+1}^\infty |w_j(t)|^2$ decays exponentially as well. 

We now obtain local exponential stability of \eqref{eq:Componentwise_equations_with_saturation_1_FW} by choosing $G'$ such that $0\in\R^n$ is in the interior of $G'$ and noting that then $0\in X$ is in the interior of $\iota(G')\times X_n^{\bot}$.

For the final statement of the proposition, pick a closed and bounded set $G\subset \intt(\iota(S) \times X_n^{\bot})$. Select $G'=\iota^{-1}\circ \pi_n (G)$, then $G'$ is a compact subset of $\intt S$, the previous computations yield
(\ref{4:juin})
for suitable constants $M$ and $a$ and for the superset $\iota(G')\times X_n^{\bot}$ which contains $G$.
\end{proof}

\begin{remark}
\label{rem:ISS_cascade_interconnection} 
It is not hard to see that
\eqref{eq:Componentwise_equations_with_saturation_infdimpart} is
input-to-state stable (ISS) with respect to the input $z$. Hence
\eqref{eq:Componentwise_equations_with_saturation_1_FW} is asymptotically
stable as a cascade interconnection of a locally asymptotically stable
system and an ISS system. However, since the general theorem on cascade
interconnections does not guarantee exponential convergence, we needed an
extra argument for Proposition~\ref{prop:Attraction_region_finite_dim_and_infinite_dim}.
\end{remark}


\begin{remark}
\label{rem:LISS_wrt_actuator_disturbances} 
Our feedback controller is robust w.r.t. additive actuator disturbances. 
Let the control input to system \eqref{closed-loop:sat} be $u(t):=\mathbf{K}z(t) + d(t)$, where 
$d \in PC(\R_+,\R^m)$ is a piecewise continuous actuator disturbance.
Then there are $r>0$,
$M_1,a_1,\gamma_1 >0$ so that for all $w(0,\cdot) \in B_r$ and all $d \in PC(\R_+,\R^m)$ with $ \sup_{s\geq 0}|d(s)|<r$ the solution of 
\begin{equation}\label{closed-loop:sat_with_disturbance}
\begin{split}
& w_t(t,x)=w_{xx}(t,x)+c(x)w(t,x)\\
&+ \sum_{k=1}^m b_k(x) \sat\big((\mathbf{K}z(t))_k+d_k(t)\big), \  t>0, \; x \in (0,L), \\
& w(t,0)=w(t,L)=0, \quad t>0, \\
& w(0,x)=w^0(x) , \quad x \in (0,L).
\end{split}
\end{equation}
satisfies the local input-to-state stability (LISS) estimate
\begin{equation}\label{4:juin2}
\|w(t,\cdot)\|_X \leq M_2e^{-a_2t} \|w(0,\cdot)\|_X + \gamma_2 \sup_{s\geq 0}|d(s)|.
\end{equation} 
This property can be obtained quite easily, since for small enough $w$ and $d$ we have 
$\sat(\mathbf{K}z(t)+d(t)) = \mathbf{K}z(t)+d(t)$, and thus
\eqref{closed-loop:sat_with_disturbance} is a linear system with a bounded
disturbance operator $d \mapsto \sum_{k=1}^m b_k(\cdot)d_k$, acting on $PC(\R_+,\R^m)$, and \eqref{4:juin2} follows as exponential stability of
\eqref{closed-loop:sat_with_disturbance} without disturbances implies LISS for a system \eqref{closed-loop:sat_with_disturbance}
with disturbances.
It is, however, much harder to estimate a region of attraction of system \eqref{closed-loop:sat_with_disturbance},
as in addition to the complexities arising in the undisturbed case the interplay between 
the size of a region of attraction and the maximal norm of a disturbance has to be analyzed. This can be an interesting topic for a future research.
For more on ISS theory of infinite-dimensional systems the reader may consult \cite{MiP19, MiW18b,TPT17} and references therein.
\end{remark}

%

\subsection{Estimate of the region of attraction for the finite-dimensional part}
\label{sec:Finite-dim_Saturated_inputs}

%
%



In view of Proposition~\ref{prop:Attraction_region_finite_dim_and_infinite_dim}, it is important to study the local exponential stability and to estimate the region of attraction of the finite-dimensional system \eqref{eq:Systfini_saturated_FW}.
We perform this task in this section.

Defining the \emph{deadzone nonlinearity} $\phi:\R^m\to\R^m$ by 
\[
\phi(u)= \sat (u)- u,\quad u\in\R^m,
\]
where $\sat$ is defined in \eqref{sat:finite:dim}, the following \emph{generalized sector condition} holds (see \cite[Lemma 1.6, Page 45]{TGS11} for a proof):
\begin{lemma}
\label{lem:generalized sector condition} 
If for some $\mathbf{C}\in\mathbb{R}^{m\times n}$,  $z \in
\mathbb{R}^n$ and \linebreak $j\in \{ 1,\ldots, m \} $ it holds that
$|((\mathbf{K}-\mathbf{C}) z)_j| \leq \ell$, then
\begin{equation*}
 \phi_j( \mathbf{K}z)  (\phi_j(\mathbf{K}z)+(\mathbf{C} z) _j)\leq 0.   
\end{equation*}
\end{lemma}
\begin{remark}
Note that other properties exist for saturation yielding to other representations of the saturation maps (see, e.g.,  \cite[Section 1.7]{TGS11}). In the present work, we insist on using generalized sector conditions because they provide constructive methods to design Lyapunov functions. 
\end{remark}
As a consequence of Lemma \ref{lem:generalized sector condition}, for all diagonal positive definite matrices
$\mathbf{D}\in\mathbb{R}^{m\times m}$, for all
$\mathbf{C}\in\mathbb{R}^{m\times n}$, and all $z \in \mathbb{R}^n$ such
that  $|((\mathbf{K}-\mathbf{C}) z)_j| \leq \ell$, $j=1\ldots, m$, we have:
\begin{equation}\label{sector:condition}
\phi( \mathbf{K}z) ^\top \mathbf{D}   (\phi(\mathbf{K}z)+ \mathbf{C}z )\leq 0.
\end{equation}

We recall the following well-known Schur complement lemma (see e.g. \cite[Chapter 2, p. 7]{BEF94}):
\begin{lemma}
\label{lem:Schur_complement} 
Let $A\in\R^{n\times n}$, $B\in\R^{m\times n}$, $C\in\R^{m\times m}$ and let $
M:=
\begin{pmatrix}
A & B^\top\\
B	& C
\end{pmatrix}
$.
If $C$ is positive definite, then $M$ is positive semidefinite if and only if its Schur complement 
$M/ C:=A-B^\top C^{-1}B$ is positive semidefinite.
\end{lemma}


The following result is only a small variation of \cite[Theorem 1]{SiT05}. Let us give a full proof using the notation used in this paper, for the sake of completeness.

%
%
\begin{proposition}
\label{prop:Stabilization_FiniteDim_Sys_with_Saturation} 
Under Assumption~\ref{ass:Stabilizability}, let $\mathbf{K}\in\R^{m\times n}$ be such that $\mathbf{A}+\mathbf{B}\mathbf{K}$ is Hurwitz.
Let 
$P\in\mathbb{R}^{n\times n}$ be symmetric positive definite, $\mathbf{D}
\in\mathbb{R}^{m\times m}$ diagonal positive definite and $\mathbf{C}\in\mathbb{R}^{m\times n}$ such that 
\begin{equation}
{\small M_1:=\left[\begin{array}{cc} 
(\mathbf{A}+ \mathbf{B} \mathbf{K})^\top P + P (\mathbf{A}+ \mathbf{B} \mathbf{K}) &  P \mathbf{B} - (\mathbf{D}\mathbf{C})^\top  \\ 
 (P \mathbf{B})^{\top} - \mathbf{D}\mathbf{C} & -2 \mathbf{D}  \end{array}\right]
< 0
}
\label{eq:LMI1}
\end{equation}
and 
\begin{eqnarray}
M_2:=\left[\begin{array}{cc}
P & (\mathbf{K}-\mathbf{C} ) ^\top
\\
 \mathbf{K}-\mathbf{C} & \ell^2 I_m
\end{array}\right]
\geq 0.
\label{eq:LMI2}
\end{eqnarray}
Then the finite-dimensional system \eqref{eq:Systfini_saturated_FW} is locally exponentially stable in $0$ with a region of attraction given by 
\begin{eqnarray}
\mathcal{A}:= \{z: \; z ^\top P z \leq 1\}.
\label{eq:Attraction_Region}
\end{eqnarray}
Moreover, in ${\cal A}$, the function $V_1$ defined by $V_1(z):= z ^\top P z$, $z\in\R^n$, decreases exponentially fast to $0$ along the solutions to \eqref{eq:Systfini_saturated_FW},
i.e. there is a constant $\alpha>0$ so that 
\begin{eqnarray}
\dot V_1 (z) \leq - \alpha |z| ^2, \quad z \in {\cal A}. 
\label{eq:Decay_Linear_Controller}
\end{eqnarray}
\end{proposition}

\begin{proof}
Let us show that \eqref{eq:LMI1}, \eqref{eq:LMI2} are feasible. As $\mathbf{A}+\mathbf{B}\mathbf{K}$ is Hurwitz, there is a symmetric positive definite matrix
$P\in\mathbb{R}^{n\times n}$ such that $(\mathbf{A}+ \mathbf{B} \mathbf{K})^\top P + P (\mathbf{A}+ \mathbf{B} \mathbf{K}) <0$ and, multiplying if needed $P$ by a positive constant, we may additionally assume that $\left[\begin{array}{cc}
P & \mathbf{K}^\top
\\
 \mathbf{K}& \ell^2 I_m
\end{array}\right]
\geq 0.
$ Then, letting $\mathbf{C}=0\in\mathbb{R}^{m\times n}$, and picking a diagonal positive definite matrix $\mathbf{D} \in\mathbb{R}^{m\times m}$ with a norm sufficiently large, we get diagonal positive definite matrix $\mathbf{D} \in\mathbb{R}^{m\times m}$ and a matrix $\mathbf{C}\in\mathbb{R}^{m\times n}$ such that (\ref{eq:LMI1}) and (\ref{eq:LMI2}) hold. 

Now pick $P, \mathbf{D}, \mathbf{C}$ such that \ref{eq:LMI1}) and (\ref{eq:LMI2}) hold and consider the Lyapunov function candidate
\[
V_1(z)= z ^\top P z,
\]
where $P\in\R^{n\times n}$ is the symmetric positive definite matrix given
by the assumptions. The time-derivative of $V_1$ along the solutions to
\eqref{eq:Systfini_saturated_FW} for $z \in \R^n$ is given by
\begin{equation*}
 \dot V_1(z) =  z ^\top [(\mathbf{A}+ \mathbf{B} \mathbf{K})^\top P + P (\mathbf{A}+ \mathbf{B} \mathbf{K})]   z + 2  z ^\top P\mathbf{B} \phi(\mathbf{K} z).   
\end{equation*}

Now assuming $|((\mathbf{K}-\mathbf{C}) z)_j| \leq \ell$ for all $j=1,\ldots,m$, it follows from (\ref{sector:condition}) that 
\begin{eqnarray*}
\dot V_1(z) &\leq&  
 z ^\top [(\mathbf{A}+ \mathbf{B} \mathbf{K})^\top P + P (\mathbf{A}+ \mathbf{B} \mathbf{K})]   z \\
&&\quad + 2  z ^\top P\mathbf{B} \phi( \mathbf{K} z) - 2 \phi( \mathbf{K}z)^\top \mathbf{D} (\phi(\mathbf{K}z)+\mathbf{C} z)\\
&=&\left[\begin{array}{c} z \\ \phi(\mathbf{K}z)\end{array}\right]^\top 
M_1    
\left[\begin{array}{c} z \\ \phi(\mathbf{K}z)\end{array}\right]
.
\end{eqnarray*}
In view of \eqref{eq:LMI1}, this implies the estimate \eqref{eq:Decay_Linear_Controller} selecting $-\alpha$ as the maximal eigenvalue of $M_1$.

It remains to ensure that $|((\mathbf{K}-\mathbf{C}) z)_j| \leq \ell$ is satisfied for all $j=1, \ldots, m$.
To do this, we use the Lyapunov function $V_1$ and we impose that the ellipsoid $\{z\in\R^n:\;  z ^\top P  z \leq 1\}$ is included in the ellipsoid $\{z\in\R^n:\; |(\mathbf{K}-\mathbf{C}) z| \leq \ell \}$ (this implies that $|((\mathbf{K}-\mathbf{C}) z)_j| \leq \ell$ holds for all \linebreak $j=1, \ldots, m$). This inclusion is equivalent to the inclusion of $\{z\in\R^n:\;  z ^\top P  z \leq 1\}$ in the set $\{z\in \R^n:  z ^\top (\mathbf{K}-\mathbf{C})^\top (\mathbf{K}-\mathbf{C}) z \leq  \ell^2\}$
which is again equivalent to the matrix inequality $ P-(\mathbf{K}-\mathbf{C} ) ^\top \frac{1}{\ell^2} (\mathbf{K}-\mathbf{C} )\geq 0$. 
As $\ell^2 I_m$ is positive definite, 
Lemma~\ref{lem:Schur_complement}
ensures that this latter matrix inequality is equivalent to \eqref{eq:LMI2}.
%
\end{proof}



\begin{remark}
We note that the formulation of
Proposition~\ref{prop:Stabilization_FiniteDim_Sys_with_Saturation}
immediately gives room for a larger estimate for the domain of attraction
so that the estimate given by $\{z: \; z ^\top P z \leq 1\}$ can never be
optimal. Indeed, in this region we have $\dot V_1 (z) \leq - \alpha |z|
^2$, so that by a continuity and compactness argument it follows that
$\dot V_1 (z)<0$ on an enlarged region of the form $\{z: \; z ^\top P z
\leq 1+\varepsilon\}$, for a suitable $\varepsilon>0$.
\end{remark}

With the previous result, it is also possible to analyze global stability. The region of
attraction is global as soon as for all $z \in\mathbb{R}^{ n}$, it holds that
$|((\mathbf{K}-\mathbf{C}) z)_j| \leq \ell$, $j=1,\ldots,m$. This is
equivalent to $\mathbf{K}=\mathbf{C}$. We thus obtain the following corollary:

\begin{corollary}
\label{cor:Global_Stabilization} 
If there exist a symmetric positive definite matrix $P\in\mathbb{R}^{n\times n}$
and a diagonal positive matrix $\mathbf{D}\in\mathbb{R}^{m\times m}$ such that  \eqref{eq:LMI1}
holds with $\mathbf{C}:=\mathbf{K}$, 
then the finite-dimensional system \eqref{eq:Systfini_saturated_FW} is globally exponentially stable in $0$. Moreover the Lyapunov function $V_1$ decreases exponentially fast to $0$ along the solutions to \eqref{eq:Systfini_saturated_FW}.
\end{corollary}
In our context, the main interest of
Proposition~\ref{prop:Stabilization_FiniteDim_Sys_with_Saturation} lies in
the following consequence for system \eqref{closed-loop:sat}.
%
\begin{theorem}{}
    \label{theo:satstab-infdim}
Consider system \eqref{closed-loop:sat} with $(\mathbf{A},\mathbf{B})$ stabilizable. Let $\mathbf{K}$ be such that 
$\mathbf{A} + \mathbf{B}\mathbf{K}$ is Hurwitz and $K$ be as in \eqref{eq:Stabilizing_controller}.
Then the closed-loop system with the feedback $K$
\begin{align}
& w_t(t,x)=w_{xx}(t,x)+c(x)w(t,x) \nonumber\\
&\qquad\qquad + \sum_{j=1}^m b_k(x) \sat\big((K \pi_n w(t,\cdot) )_k\big),  \ t>0, \; x \in (0,L), \nonumber\\
& w(t,0)=w(t,L)=0, \quad t>0, \\
& w(0,x)=w^0(x) , \quad x \in (0,L).    \nonumber
\end{align}
is locally exponentially stable in $0$ with region of attraction $\imath({\cal A})
\times X_n^\perp$. In addition, the constants of decay can be chosen
uniformly on $\imath({\cal A})
\times X_n^\perp$.
\end{theorem}

\begin{proof}
    Exponential stabilization and the region of attraction are immediate
    consequences of
    Proposition~\ref{prop:Stabilization_FiniteDim_Sys_with_Saturation} in
    combination with
    Proposition~\ref{prop:Attraction_region_finite_dim_and_infinite_dim}. As
    the exponential estimate can be chosen uniformly on the set ${\cal A}$,
    the proof of
    Proposition~\ref{prop:Attraction_region_finite_dim_and_infinite_dim}
    shows that the uniformity holds on all of $\imath({\cal A}) \times
    X_n^\perp$.
\end{proof}


\subsection{Numerical implementation of Proposition~\ref{prop:Stabilization_FiniteDim_Sys_with_Saturation}}
\label{sec:Numerical implementation}

The inequalities \eqref{eq:LMI1} and \eqref{eq:LMI2} are \emph{bilinear} matrix inequalities, since they have 
the bilinear cross term $\mathbf{D}\mathbf{C}$ in the unknowns $\mathbf{D}$ and $\mathbf{C}$. The nonlinearity complicates the numerical analysis of \eqref{eq:LMI1} and \eqref{eq:LMI2}. However, they can be reformulated into equivalent \emph{linear matrix inequalities (LMIs)}, which are easier to solve numerically.

Let $S= P^{-1}$ and $ \mathbf{E} = \mathbf{D} ^{-1}$. 
Since $M_1<0$ if and only if 
\[
\left[\begin{array}{cc} S& 0 \\ 0 &  \mathbf{E} \end{array}\right]^\top
M_1
\left[\begin{array}{cc} S& 0 \\ 0 &  \mathbf{E} \end{array}\right]<0,
\]
we obtain that (\ref{eq:LMI1})  is equivalent to
\begin{equation}
{\small \left[\begin{array}{cc} 
S (\mathbf{A}+ \mathbf{B} \mathbf{K})^\top +  (\mathbf{A}+ \mathbf{B} \mathbf{K})S  &   \mathbf{B} \mathbf{E}  - S \mathbf{C}^\top  \\ 
 ( \mathbf{B} \mathbf{E}  - S \mathbf{C}^\top )^\top & -2 \mathbf{E}  \end{array}\right]
< 0.
}
\end{equation}
Now pre- and post-multiplying (\ref{eq:LMI2}) by $\left[\begin{array}{cc} S& 0 \\ 0 &   I_m
\end{array}\right]$, we get that (\ref{eq:LMI2})  is equivalent to 
\begin{eqnarray}
\left[\begin{array}{cc}
S &S  \mathbf{K} ^\top- S \mathbf{C}  ^\top
\\
(S  \mathbf{K} ^\top- S \mathbf{C}  ^\top
)^\top  & \ell^2 I_m
\end{array}\right]
\geq 0.
\end{eqnarray}
Letting $ \mathbf{Y}= S \mathbf{C}  ^\top$, we obtain (see also \cite[Theorem 1]{SiT05} in a different context):
\begin{proposition}
\label{prop:BMI-to-LMI} 
Existence of a symmetric positive definite matrix
$P\in\mathbb{R}^{n\times n}$, a diagonal positive definite matrix $\mathbf{D} \in\mathbb{R}^{m\times m}$ and a matrix $\mathbf{C}\in\mathbb{R}^{m\times n}$ such that (\ref{eq:LMI1}) and (\ref{eq:LMI2}) hold is equivalent to 
 the existence of a symmetric positive definite matrix
$S\in\mathbb{R}^{n\times n}$, a diagonal positive definite matrix $\mathbf{E} \in\mathbb{R}^{m\times m}$ and $\mathbf{Y}\in\mathbb{R}^{m\times n}$ for which it holds that
\begin{equation}\label{eq:LMI1:eq}
{\small \left[\begin{array}{cc} 
S (\mathbf{A}+ \mathbf{B} \mathbf{K})^\top +  (\mathbf{A}+ \mathbf{B} \mathbf{K})S  &   \mathbf{B} \mathbf{E}  - \mathbf{Y}  \\ 
 ( \mathbf{B} \mathbf{E}  - \mathbf{Y}^\top )^\top & -2 \mathbf{E}  \end{array}\right]
< 0
}
\end{equation}
and 
\begin{eqnarray}\label{eq:LMI2:eq}
\left[\begin{array}{cc}
S &S  \mathbf{K} ^\top-  \mathbf{Y}  ^\top
\\
(S  \mathbf{K} ^\top-  \mathbf{Y}  ^\top
)^\top  & \ell^2 I_m
\end{array}\right]
\geq 0.
\end{eqnarray}
\end{proposition}
This reformulation is important as the constraints (\ref{eq:LMI1:eq}) and  (\ref{eq:LMI2:eq}) are linear in the variables $S$, $\mathbf{E}$ and $\mathbf{Y}$, in contrast to bilinear inequalities (\ref{eq:LMI1}) and  (\ref{eq:LMI2}).

Finally, no structure is imposed on both the matrix $\mathbf{C}\in\mathbb{R}^{m\times n}$ and on the matrix $\mathbf{Y}\in\mathbb{R}^{m\times n}$, and we let $ \mathbf{C}= (S^{-1} \mathbf{Y})  ^\top$ from the knowledge of $S$ and $\mathbf{Y}$ (and respectively we may let $ \mathbf{Y}= P^{-1} \mathbf{C}  ^\top$ from the knowledge of $P$ and $ \mathbf{C} $).

\subsection{Scalar control inputs}
\label{sec:Reformulation_of_MIs}

In this section we specialize our results for scalar control inputs, i.e. $m=1$, which allows for a numerically more efficient method. We need the following lemma.
\begin{lemma}
\label{lem:Pdf_matrix_properties} 
Consider a symmetric matrix $P = 
\begin{pmatrix}
P_{11}	& P_{12}\\
P_{21}	& P_{22}
\end{pmatrix}
\in\R^{(n_1+n_2)\times(n_1+n_2)}$, where $P_{ij}$ are matrices of appropriate dimension and $P_{11}>0$, $P_{22}>0$.
Define
$
P_a:=
\begin{pmatrix}
a P_{11}	& P_{12}\\
P_{21}	& P_{22}
\end{pmatrix}
$	
for
 $a\geq 0$. Then there exists $a^*\geq 0$:
\begin{itemize}
	\item[(i)] $P_{a^*}\geq 0$, and $P_{a^*}$ is not positive definite,
	\item[(ii)] $P_{a}> 0$ for all $a>a^*$,
	\item[(iii)] for all $a\in[0,a^*)$ the matrix $P_{a}$ is not positive semidefinite.		
\end{itemize}
\end{lemma}

\begin{proof}
    As $P_{22}>0$, the matrix
    $P_a$ is positive definite if and only if the Schur complement
    \begin{equation*}
        P_a/P_{22} = aP_{11} - P_{12} P_{22}^{-1} P_{21} 
    \end{equation*}
    is positive definite, see Lemma~\ref{lem:Schur_complement}. Note that as $(P_{12})^\top=P_{21}$ and $(P_{22}^{-1})^\top = P_{22}^{-1}$, we have that
\[
(P_{12} P_{22}^{-1} P_{21})^\top = P^\top_{12} (P_{22}^{-1})^\top (P_{21})^\top = P_{12} P_{22}^{-1} P_{21},
\]
		and as $P_{11}$ is positive definite, $P_a/P_{22}$ is a symmetric matrix.
    		
		Weyl's inequality (see \cite[Chapter IV, Corollary 4.9, p. 203]{StS90}) implies that all eigenvalues of
    $P_a/P_{22}$ are strictly increasing functions of $a$ with a slope
    bounded below everywhere by 
    $\lambda_{\mathrm{min}}(P_{11})>0$. The claim is now immediate. Clearly, $P_0$ is not
    positive definite. As all eigenvalues are strictly increasing with an
    affine lower bound, eventually at some $a^*\geq 0$ the smallest eigenvalue is equal to 0
    by continuity of eigenvalues. For all $a> a^*$ all
    eigenvalues are positive.
\end{proof}

%
%

The main result of this section is
\begin{proposition}
\label{prop:Stabilization_FiniteDim_Sys_with_Saturation_modified} 
Under Assumption~\ref{ass:Stabilizability}, let $\mathbf{K}\in\R^{m\times n}$ be such that $\mathbf{A}+\mathbf{B}\mathbf{K}$ is Hurwitz.
Let 
$\tilde{P}\in\mathbb{R}^{n\times n}$ be symmetric positive definite and $\mathbf{C}\in\mathbb{R}^{1\times n}$ such that 
\begin{equation}
{\small
\hspace{-3mm}\tilde{M}_1{:=}\left[\begin{array}{cc} 
(\mathbf{A}+ \mathbf{B} \mathbf{K})^\top \tilde{P} + \tilde{P} (\mathbf{A}+ \mathbf{B} \mathbf{K}) &  \tilde{P} \mathbf{B} - \mathbf{C}^\top  \\ 
 (\tilde{P} \mathbf{B})^{\top} - \mathbf{C} & -2 \end{array}\right]
\hspace{-0.5mm}< 0.
}
\label{eq:LMI1_modified}
\end{equation}
Then the finite-dimensional system \eqref{eq:Systfini_saturated_FW} is locally exponentially stable in $0$ with a region of attraction given by
\begin{eqnarray}
\{z: \; z ^\top \tilde{P} z \leq \mathbf{D}^ {-1} \},
\label{eq:Attraction_Region_D}
\end{eqnarray}
where $\mathbf{D}$ is the minimal real number such that:
\begin{eqnarray}
\tilde{M}_2:=\left[\begin{array}{cc}
\mathbf{D}\tilde{P} & (\mathbf{K}-\mathbf{C} ) ^\top
\\
 \mathbf{K}-\mathbf{C} & \ell^2
\end{array}\right]
\geq 0.
\label{eq:LMI2_modified}
\end{eqnarray}
Moreover, in this region of attraction, the function $V_1$ defined by
$V_1(z)= z ^\top P z := \mathbf{D}  z ^\top \tilde{P} z$, for all $z\in\R^n$, decreases exponentially fast to $0$ along the solutions to \eqref{eq:Systfini_saturated_FW},
i.e. there is a constant $\alpha>0$ so that 
\begin{eqnarray}
\dot V_1 (z) \leq - \alpha |z| ^2. 
\label{eq:Decay_Linear_Controller_MI_reform}
\end{eqnarray}
\end{proposition}

\begin{proof}
Define $P:=\mathbf{D}\tilde{P}$. Substituting this expression into 
\eqref{eq:LMI1} and \eqref{eq:LMI2} and multiplying it by $\mathbf{D}^{-1}$ to obtain the inequalities
\eqref{eq:LMI1_modified} and \eqref{eq:LMI2_modified}.
If there exist $\tilde{P}$ and $\mathbf{C}$ as in Proposition~\ref{prop:Stabilization_FiniteDim_Sys_with_Saturation_modified} so that the LMI \eqref{eq:LMI1_modified} holds, then according to Lemma~\ref{lem:Pdf_matrix_properties}, one can find $\mathbf{D}>0$ so that 
\eqref{eq:LMI2_modified} holds as well.

The estimate for the region of attraction follows from Proposition~\ref{prop:Stabilization_FiniteDim_Sys_with_Saturation}.
We choose the minimal $\mathbf{D}$ to obtain the largest estimate of a region of attraction (for given $\tilde{P}$ and $\mathbf{C}$).
\end{proof}

\begin{remark}
\label{rem:Attraction_Regions} 
Lemma~\ref{lem:Pdf_matrix_properties} can be helpful in finding the
optimal $\mathbf{D}$ as it shows that the problem is to find the unique
root of a piecewise analytic function.
\end{remark}


\subsection{Enlarging the region of attraction using a dynamic controller}
\label{sec:anti-windup}

Consider a
dynamic controller for system \eqref{eq:Systfini_saturated_FW}, instead of a static
controller. By doing so, we add some degrees of freedom and thus we may enlarge the region of attraction. 
Such approach is useful for many control systems (see e.g. \cite[Example 8.1]{TGS11} for a simple $2\times 2$ example). 
To be more specific, we consider 
an additional finite-dimensional state $z_c$ in $\R^{n_c}$ (for a given integer $n_c$) and design matrices $\mathbf{K}_1$, $\mathbf{K}_2$, $\mathbf{A}_c$ and $\mathbf{B}_c$ 
of appropriate dimensions so that the (estimation of the) region of attraction of 
\begin{equation}
\label{finite:dim:sat:augmente}
\begin{array}{rcl}
\dot z&=& \mathbf{A}z + \mathbf{B} \sat (\mathbf{K}_1 z+\mathbf{K}_2 z _c)
\\
\dot z_c&=& \mathbf{A}_1z_c + \mathbf{A}_2  z 
\end{array}
\end{equation}
is larger (in the $z$-direction) than the estimate provided by Proposition \ref{prop:Stabilization_FiniteDim_Sys_with_Saturation} for \eqref{eq:Systfini_saturated_FW}. We rewrite the dynamics (\ref{finite:dim:sat:augmente}) by 
\begin{equation}
\label{finite:dim:sat:augmente_matrix_form}
\begin{array}{rcl}
\dot Z&=& \bar{\mathbf{A}}Z +  \bar{\mathbf{B}} \sat (\bar{\mathbf{K}} Z  )
\end{array}\end{equation}
where 
\begin{equation}
\label{def:bar:matrices}
\bar{\mathbf{A}}= \begin{pmatrix}\mathbf{A} &0 \\ \mathbf{A}_2 & \mathbf{A}_1\end{pmatrix},\; 
\bar{\mathbf{B}}=\begin{pmatrix} \mathbf{B} \\ 0\end{pmatrix}, \;
\bar{\mathbf{K}}=\begin{pmatrix} \mathbf{K}_1 & \mathbf{K}_2\end{pmatrix},
\end{equation}
with $\bar{\mathbf{A}}\in\R^{(n+n_c)\times (n+n_c)}$, $\bar{\mathbf{B}}\in\R^{(n+n_c)\times m}$ and \linebreak
$\bar{\mathbf{K}}\in\R^{m\times (n+n_c)}$.

\smallskip{}

\begin{remark}
\label{rem:Ellipsoids} 
For given symmetric positive definite matrices $P \in
\R^{n\times n}$ and $\bar P \in \R^{(n+n_c)\times (n+n_c)}$, the inclusion
of the ellipsoid $\{z\in \R^n : \; z^\top Pz \leq 1\}$ in
the projection (onto $\R^n$) of the ellipsoid $\{Z \in \R^{n+n_c}: \; Z^\top \bar PZ \leq 1\}$ is equivalent to 
\begin{equation}
\label{eq:Ellipsoids-inclusion}
\begin{pmatrix} I_n & 0\end{pmatrix} \bar P \begin{pmatrix} I_n \\ 0\end{pmatrix} - P \leq 0\ .    
\end{equation}
\end{remark}

Using this remark  and applying Proposition \ref{prop:Stabilization_FiniteDim_Sys_with_Saturation} to system (\ref{finite:dim:sat:augmente}), we have the following:

\begin{proposition}
\label{prop:Stabilization_FiniteDim_Sys_with_Saturation:bis} 
Consider system \eqref{eq:Systfini_saturated_FW} with a stabilizable pair $(\mathbf{A},\mathbf{B}) \in \R^{n\times n}\times\R^{n\times m}$, and matrices $\bar{\mathbf{A}}$, $\bar{\mathbf{B}}$ and $\bar{\mathbf{K}}$ defined in (\ref{def:bar:matrices}).
%
%
Choose a symmetric positive definite matrix $\bar P\in\mathbb{R}^{(n+n_c)\times (n+n_c)}$, a diagonal positive matrix $\bar{\mathbf{D}} \in\mathbb{R}^{m\times m}$ and $\bar{\mathbf{C}}\in\mathbb{R}^{m\times (n+n_c)}$ such that \eqref{eq:Ellipsoids-inclusion} holds and
\begin{eqnarray}
\left[\begin{array}{cc} 
(\bar{\mathbf{A}}+ \bar{\mathbf{B}} \bar{\mathbf{K}})^\top \bar P + \bar P (\bar{\mathbf{A}}+ \bar{\mathbf{B}} \bar{\mathbf{K}}) &  \bar P \bar{\mathbf{B}} - (\bar{\mathbf{D}}\bar{\mathbf{C}})^\top  \\ 
 (\bar P \bar{\mathbf{B}})^{\top} - \bar{\mathbf{D}}\bar{\mathbf{C}} & -2 \bar{\mathbf{D}}  \end{array}\right]
< 0,
\label{eq:LMI1:bis}
\end{eqnarray}
\vspace{-2mm}
\begin{eqnarray}
\left[\begin{array}{cc}
\bar P & (\bar{\mathbf{K}}-\bar{\mathbf{C}} ) ^\top  \\
 \bar{\mathbf{K}}-\bar{\mathbf{C}} & \ell^2 I_m
\end{array}\right]
\geq 0.
\label{eq:LMI2:bis}
\end{eqnarray}
Then the finite-dimensional system (\ref{finite:dim:sat:augmente}) is
locally exponentially stable in $0$ with a region of attraction given by
\begin{center}
${\cal A}_{\bar {P} } :=\{Z \in \R^{n + n_c} : \; Z ^\top \bar P Z \leq 1\}$. 
\end{center}

Moreover, the projection of the ellipsoid $\{Z : \; Z ^\top \bar P Z \leq
1\}$ onto $\R^n$ is larger than the region of attraction ${\cal A}$ given
by Proposition \ref{prop:Stabilization_FiniteDim_Sys_with_Saturation}
in \eqref{eq:Attraction_Region} for  system \eqref{eq:Systfini_saturated_FW}.

Finally, in ${\cal A}_{\bar {P} }$, the function $V_1$ defined by 
\[
\overline  V_1(Z)= Z ^\top \bar P Z, \quad  Z\in\R^{n+n_c},
\] 
decreases exponentially fast to $0$ along the solutions to (\ref{finite:dim:sat:augmente}),
i.e. there is a constant $\alpha>0$ so that for all $Z \in {\cal A}_{\bar {P} }$
\begin{eqnarray}
\dot {\overline {V_1}}(Z) \leq - \alpha |Z| ^2. 
\end{eqnarray}
\end{proposition}

As in Section~\ref{sec:Numerical implementation}, the inequalities \eqref{eq:LMI1:bis} and \eqref{eq:LMI2:bis} can be reformulated as linear matrix inequalities.

\subsection{Lyapunov analysis of the closed-loop system under saturation control}


Now under the assumptions of Proposition \ref{prop:Stabilization_FiniteDim_Sys_with_Saturation}, we provide a
Lyapunov function $V:\ell_2(\N^*,\R)\to\R_+$ for the system (\ref{eq:Componentwise_equations_with_saturation_1_FW}). To this end,
we use the Lyapunov function $P$ provided by Proposition
\ref{prop:Stabilization_FiniteDim_Sys_with_Saturation} for the
finite-dimensional subsystem \eqref{eq:Systfini_saturated_FW}.
\begin{proposition}
\label{prop:Attraction_region_finite_dim_and_infinite_dim_Lyapunov} 
Consider system \eqref{eq:Componentwise_equations_with_saturation_1_FW} and
assume $\mathbf{K}$ is chosen such that the subsystem \eqref{eq:Systfini_saturated_FW} is locally
exponentially stable in $0$. Assume further that $P \in \R^{n \times n}$ is
symmetric positive definite and such that for $\tilde V(z) = z^\top P z, z \in
\R^{n}$ we have
$\dot {\tilde {V } } (z)\leq - \alpha | z |^2$ on the set ${\cal A}:=\{ z \in \R^n: \tilde V(z) \leq 1 \}$
along the solutions of \eqref{eq:Systfini_saturated_FW}. 
Then there exist $\gamma, C > 0$ such that for $Q: \ell_2 (\N^*, \R) \to
\R$ defined by $Q(w) := \sum_{j=n+1}^\infty w_j^2$ the function
\begin{equation}
\label{def:Lyapunov}
  V(w) := \tilde V(\pi_n w) + \gamma Q(w) = z^\top P z + \gamma Q(w)
\end{equation}
(with the identification $\pi_n w =z$) satisfies for all \linebreak $w \in {\cal A}
\times \ell_{2,j>n}$ that along the solutions of \eqref{eq:Componentwise_equations_with_saturation_1_FW}
\begin{equation*}
  \dot V(w)\leq - C \| w \|^2_{\ell_2}.
\end{equation*}
In particular, it follows that $S \times \ell_{2,j>n}$ is a region of attraction of
$0$ for system \eqref{eq:Componentwise_equations_with_saturation_1_FW}.
\end{proposition}

\begin{proof}
   We write $w\in \ell_2 (\N^*, \R)$ as $w=z + z^\bot$, where $z\in \R^n,
   z^\bot\in \ell_{2,j>n}$.
	Here we identify $\R^n$ with the sequences with
support in $\{1, \ldots,n \}$ and $\ell_{2,j>n}$ is the set of sequences
in $\ell_2(\N^*, \R)$ which are $0$ in the first $n$ entries. 
	
	This decomposition is unique.
   Due to (\ref{refeta}), we have along the solutions to (\ref{eq:Componentwise_equations_with_saturation_1_FW})
\begin{align*}
\dot Q(w) 
 = & 2\sum_{j=n+1}^\infty w_j(t)\dot{w}_j(t)\\
=& 2\sum_{j=n+1}^\infty w_j(t)\Big(\lambda_j w_j(t) + \mathbf{b}_j \cdot \sat(\mathbf{K}z(t))\Big)\\
\leq& -2\eta\sum_{j=n+1}^\infty w^2_j(t)  + 2\sum_{j=n+1}^\infty |w_j(t)| |\mathbf{b}_j| |\sat(\mathbf{K}z(t))|\\
\leq& -2\eta\sum_{j=n+1}^\infty w^2_j(t)  + 2\sum_{j=n+1}^\infty |w_j(t)| |\mathbf{b}_j| \|\mathbf{K}\||z(t)|.
\end{align*}
Cauchy-Bunyakovsky-Schwarz inequality implies that
\begin{align*}
\dot Q(w) \leq& -2\eta \|z^\bot(t)\|_{\ell_2}^2  + 2 \|z^\bot(t)\|_{\ell_2} \|b^\bot(t)\|_{\ell_2} \|\mathbf{K}\||z(t)|.
\end{align*}
Using Young's inequality we have for all $\kappa>0$ that $2
\|z^\bot(t)\|_{\ell_2}|z(t)| \leq \kappa \|z^\bot(t)\|_{\ell_2}^2 +
\frac{1}{\kappa}|z(t)|^2$. We proceed to:
\begin{align*}
\dot Q(w) \leq &- \big(2\eta-\kappa\|b^\bot\|_{\ell_2}  \|\mathbf{K}\|\big) \| z^\bot(t)\|^2_{\ell_2} +\frac{1}{\kappa}\|b\|_{\ell_2}  \|\mathbf{K}\| | z(t)|^2.
\end{align*}
Therefore for any choice of $\gamma$ in \eqref{def:Lyapunov}, we have by an application of Proposition
\ref{prop:Stabilization_FiniteDim_Sys_with_Saturation} along the solutions to
(\ref{eq:Componentwise_equations_with_saturation_1_FW}) that
\begin{align*}
\dot{V}(w) \leq & -\big(\alpha - \frac{\gamma}{\kappa}\|b^\bot\|_{\ell_2}  \|\mathbf{K}\|\big) | z(t)|^2\\
&\qquad\qquad- \gamma \big(2\eta-\kappa\|b^\bot\|_{\ell_2} \|\mathbf{K}\|\big) \| z^\bot(t)\|^2_{\ell_2}.
\end{align*}
Therefore selecting first $\kappa>0$ such that $2\eta-\kappa\|b^\bot\|_{\ell_2} \|\mathbf{K}\|>0$ and then selecting $\gamma>0$ such that 
$\alpha - \frac{\gamma}{\kappa}\|b^\bot\|_{\ell_2} \|\mathbf{K}\|>0$, we get the existence of $C>0$ such that
\begin{equation}
\label{dot:V:first:PDE}
\begin{array}{rcl}
\dot{V}(w)&\leq & -C | z(t)|^2 - C\| z^\bot(t)\|^2_{\ell_2}
\end{array}
\end{equation}
provided that $z(t)$ lies in the ellipsoid ${\cal A}$, whatever the value
of $z^\bot(t)$. As ${\cal A}\times \ell_{2,j>n}$ is an invariant subset of $\ell_2$ for a system (\ref{eq:Componentwise_equations_with_saturation_1_FW}), we obtain that $V$ is a Lyapunov function for (\ref{eq:Componentwise_equations_with_saturation_1_FW}), with a guaranteed region of attraction containing
${\cal A}\times \ell_{2,j>n}$. 
\end{proof}

Note that this provides an alternative proof of Proposition~\ref{prop:Attraction_region_finite_dim_and_infinite_dim}. We can also use Proposition \ref{prop:Stabilization_FiniteDim_Sys_with_Saturation:bis} and obtain a similar result for a dynamical controller.

\section{Numerical experiments}
\label{sec:Numerics}

In this section we use Proposition~\ref{prop:Stabilization_FiniteDim_Sys_with_Saturation_modified} to obtain estimates for a region of attraction for the unstable heat equation \eqref{closed-loop:sat} subject to a saturated scalar feedback controller.
Let $c(\cdot)$ in equation \eqref{closed-loop:sat} be a constant
function. With slight abuse of notation we will write $c(\cdot) = c = const$.

According to \cite[pp. 16-17]{Hen81} the eigenvalues of the operator
\begin{equation}
A:=\partial_{xx}+c\;\mathrm{id}: X \to X
\end{equation}
on the domain $D(A)=H^2(0,L)\cap H^1_0(0,L)$
are given by
\begin{eqnarray}
\lambda_j:= - \frac{\pi^2}{L^2} j^2 + c,\quad j\in\N^*,
\label{eq:EV_of_simple_A}
\end{eqnarray}
and the eigenfunctions $e_j$, $j\in\N^*$ of $(A,D(A))$, which form a basis of $L_2(0,1)$ are given by
\begin{eqnarray}
e_j(x):= \Big( \frac{2}{L}\Big)^{1/2} \sin\frac{j\pi x}{L},\quad j\in\N^*, \quad x\in(0,L).
\label{eq:Efunctions_of_simple_A}
\end{eqnarray}



Next we estimate the region of attraction of system \eqref{closed-loop:sat} for the following choice of parameters:
\[
c(x) \equiv 10,\quad L = 2, \quad \ell = 2,\quad b = e_1 + e_2,
\]
where the $e_j$ are defined in \eqref{eq:Efunctions_of_simple_A}.
For these parameters there are only 2 unstable eigenvalues. Using the eigenfunction decomposition of solutions of \eqref{closed-loop:sat} as in Section~\ref{sec:Start}, we see that the matrices $\mathbf{A}, \mathbf{B}$ defined in \eqref{eq:def:A1} have the form
\[
\mathbf{A} \approx
\begin{pmatrix}
7.5325989 &  0       \\       	
0        &  0.1303956
\end{pmatrix}
,
\qquad
\mathbf{B} = 
\begin{pmatrix}
1  \\       	
1  
\end{pmatrix}
.
\]
As the diagonal entries in the matrix $\mathbf{A}$ are distinct, and all
the components of $\mathbf{B}$ are nonzero, the 
pair $(\mathbf{A},\mathbf{B})$ is stabilizable.
Different choices of the matrix $\mathbf{K}$ for the stabilizing feedback
$u(t)=\mathbf{K}z(t)$ lead to different attraction rates and different
regions of attraction. We demonstrate this with two examples.

\begin{figure*}%
\centering	
\IfFileExists{Simulation_b_is_1_1_ell_attr_diverg_D_0_01_poles_-1_black_and_red_cropped.png}{%
\includegraphics[scale=0.4]{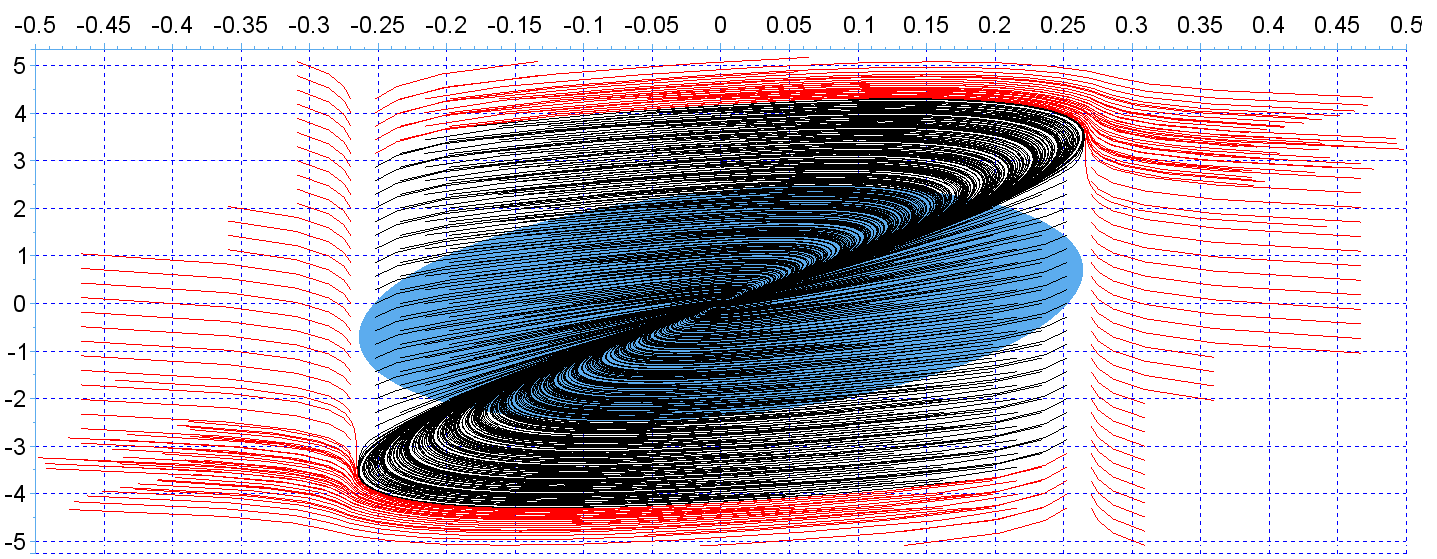}%
}
{Fabian does not have this particular figure} 
\caption{Region of attraction of \eqref{eq:Systfini_saturated_FW} (in blue) for the choice \eqref{eq:Kmatr_Sim1},
  \eqref{eq:PCD_Sim1}, computed via the LMI technique. Trajectories
  of \eqref{eq:Systfini_saturated_FW} are computed by direct solution of the ODE,
  trajectories attracted to the origin are in black, diverging
  trajectories are in red.}%
\label{fig:Attraction_Region_Traj_Ell_Diverg_Sim1}%
\end{figure*}

\begin{figure*}%
\centering
\hspace{-2mm}\includegraphics[scale=0.4]{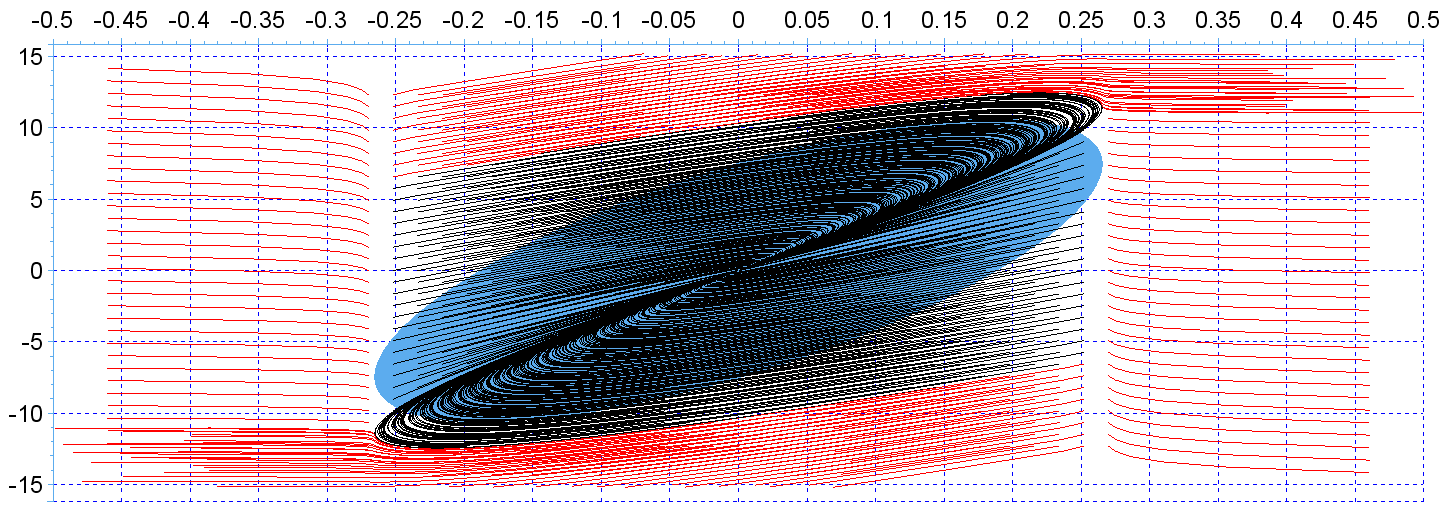}%
\caption{Region of attraction of \eqref{eq:Systfini_saturated_FW} (in blue) for the choice
  \eqref{eq:Kmatr_Sim2}, \eqref{eq:PCD_Sim2}, computed via the LMI
  technique. Trajectories of \eqref{eq:Systfini_saturated_FW} are computed by direct
  solution of the ODE, trajectories attracted to the origin are in black,
  diverging trajectories are in red.}%
\label{fig:Attraction_Region_Traj_Ell_Diverg_Sim2}
\end{figure*}

\subsection{Choice 1: Placing the poles at $(-1,-1)$}
\label{sec:Poles-1}

First we choose the matrix $\mathbf{K}$ so that $\sigma(\mathbf{A} + \mathbf{B}\mathbf{K}) = \{-1\}$, which results in 
\begin{eqnarray}
\mathbf{K} \approx
\begin{pmatrix}
-9.835618  & 0.1726235
\end{pmatrix}
.
\label{eq:Kmatr_Sim1}
\end{eqnarray}
To estimate a region of attraction of \eqref{eq:Systfini_saturated_FW} by Proposition~\ref{prop:Stabilization_FiniteDim_Sys_with_Saturation_modified}, we proceed in two steps.
\begin{itemize}
	\item[(i)] First we solve the inequality
          \eqref{eq:LMI1_modified} together with the additional constraints: $\tilde{P}-{\tilde{P}}^\top=0$ and ${\tilde{P}}>0$.
Additionally, we impose an optimality condition for $\mathbf{C}$:
\begin{eqnarray}
(\mathbf{K}-\mathbf{C})\cdot(\mathbf{K}-\mathbf{C})^\top \to \min,
\label{eq:Additional_Optimisation_Condition}
\end{eqnarray}
where $\cdot$ is the scalar product of vectors.

	\item[(ii)] The idea behind \eqref{eq:Additional_Optimisation_Condition}
 is to minimize the off-diagonal elements of the matrix $\tilde{M}_2$, which gives us at the second step 
the possibility to find large $\mathbf{D}$ (optimal for the given $\tilde{P}$, $\mathbf{C}$) satisfying the bilinear matrix inequality \eqref{eq:LMI2_modified}.
\end{itemize}

This algorithm is implemented in Scilab. For solution of the LMI \eqref{eq:LMI1_modified} the LMITOOL package has been used.
The resulting matrices $\tilde{P}, \mathbf{C}$ and the real number
$\mathbf{D}$ are (approximately):
\begin{subequations}
\begin{align}
\tilde{P} =& 
\begin{pmatrix}
   2.1277468 & -0.0655569 \\
  -0.0655569 &  0.0243008
\end{pmatrix},\\
\mathbf{C} =& (  -2.0635579,  0.0844904)
,\quad
 \mathbf{D} =  7.359375.
\end{align}
\label{eq:PCD_Sim1}
\end{subequations}


Figure~\ref{fig:Attraction_Region_Traj_Ell_Diverg_Sim1} shows an elliptic region of attraction \eqref{eq:Attraction_Region},
subject to $\tilde{P}, \mathbf{C}, \mathbf{D}$ given by \eqref{eq:PCD_Sim1} (in
blue).
Furthermore, in the same figure some trajectories are depicted
obtained through direct simulation of \eqref{eq:Systfini_saturated_FW}. The
trajectories in black tend to the origin while those in
red are diverging. 
This provides an approximation of the maximal region of attraction of \eqref{eq:Systfini_saturated_FW}. 
It can be seen that in one direction the ellipsoid obtained by our method
approximates the actual region of attraction very well, but the results
are not tight in the orthogonal direction.

Using Theorem~\ref{theo:satstab-infdim}, we obtain an approximation of the region of attraction of the PDE \eqref{closed-loop:sat} subject to the feedback \eqref{eq:Stabilizing_controller}.

%

\begin{remark}
\label{rem:Optimality_Conditions} 
\textbf{(Importance of the optimality conditions)} Different solutions
$\tilde{P}$, $\mathbf{C}$, $\mathbf{D}$ of the matrix inequalities
\eqref{eq:LMI1_modified}, \eqref{eq:LMI2_modified} lead to very different
estimates of a region of attraction \eqref{eq:Attraction_Region} of the
model \eqref{eq:Systfini_saturated_FW}.  Thus, it is important to pick solutions
resulting in as large regions of attraction as possible. Here we chose a solution satisfying the optimality condition
\eqref{eq:Additional_Optimisation_Condition}.  Enforcing further
optimality conditions may provide different estimates. The union of
regions of attractions is again a region of attraction and so the maximal
region can be explored further by solving
\eqref{eq:LMI1_modified}, \eqref{eq:LMI2_modified} for different
optimality conditions.

In the two dimensional case, one option is to fix the eigendirections
of $P$ and to optimize the eigenvalues to get an idea of the extension of
the maximal domain of attraction in particular directions.
Other size criteria exist for optimization of domains of attraction, as volume maximization and trace minimization. See \cite[Section 2.2.5]{TGS11} for more details and a complete introduction of such optimization problems.
\end{remark}

\begin{remark}
\label{rem:Computational_Costs} 
\textbf{(Computational cost)} The time needed to solve the problem was (on a system with the specs: Intel(R) Core(TM) i5-3317U \@ 1.70GHz, 16 GB RAM, Windows 10):
\begin{itemize}
	\item Finding $\tilde{P}$, $\mathbf{C}$, $\mathbf{D}$ via LMIs: 0.0166561 seconds.
	\item Plotting the obtained region: 0.0071209 seconds.
	\item Time for solving the ODE \eqref{eq:Systfini_saturated_FW} for $31^2=961$ distinct initial conditions on the time-interval $[0,10]$ on a grid consisting of 100 points and for plotting the resulting trajectories: 43.359664 seconds.
\end{itemize}
This shows the computational efficiency of our method.
\end{remark}

\subsection{Choice 2: Putting the poles to $\{-0.1,-0.2\}$}
\label{sec:Poles_-01_-02}

Now let us choose the matrix $\mathbf{K}$ so that $\sigma(\mathbf{A} + \mathbf{B}\mathbf{K}) = \{-0.1,-0.2\}$. This choice makes the attraction rate of the closed-loop system much slower, than in the previous simulation. This has however, some advantages, as we will see next. The resulting matrix $\mathbf{K}$ is:
\begin{eqnarray}
\mathbf{K} = 
\begin{pmatrix}
-7.9732782 &  0.0102837
\end{pmatrix}
.
\label{eq:Kmatr_Sim2}
\end{eqnarray}

As in the previous simulation, we solve the LMI \eqref{eq:LMI1_modified} subject to the additional optimality condition
\eqref{eq:Additional_Optimisation_Condition}. The corresponding matrices $\tilde{P}, \mathbf{C}, \mathbf{D}$ are:
\begin{subequations}
\begin{align}
\tilde{P} =& 
\begin{pmatrix}
   0.3108695 & -0.0054849\\
  -0.0054849 &  0.000195 
\end{pmatrix},	\\
\mathbf{C} =& (-0.3053879, 0.0054754)
,\quad
\mathbf{D} =  90.625.
\end{align}
\label{eq:PCD_Sim2}
\end{subequations}
As we see, with the choice of the stabilizing feedback \eqref{eq:Kmatr_Sim2}, the region of attraction becomes significantly larger, although at the cost of reducing the rate of convergence of the trajectories to the origin.
Furthermore, the choice of the matrices \eqref{eq:PCD_Sim2} leads to a
better estimate of the region of attraction, in comparison to the
situation in Section~\ref{sec:Poles-1}. This can be seen by comparing the Figures~\ref{fig:Attraction_Region_Traj_Ell_Diverg_Sim1}
and \ref{fig:Attraction_Region_Traj_Ell_Diverg_Sim2}.

\begin{remark}
\label{rem:Computational_Costs_2} 
\textbf{(Computational costs)} For this problem the elapsed time is (on a system with the specs: Intel(R) Core(TM) i5-3317U \@ 1.70GHz, 16 GB RAM, Windows 10)
\begin{itemize}
	\item Finding $\tilde{P}$, $\mathbf{C}$, $\mathbf{D}$ via LMIs: 0.018131 seconds.
	\item Plotting the obtained region: 0.0129341 seconds.
	\item Time for solving the ODE \eqref{eq:Systfini_saturated_FW} for $31^2=961$ distinct initial conditions on the time-interval $[0,60]$ on a grid consisting of 600 points and for the plotting of the resulting trajectories: 91.802833 seconds.
\end{itemize}
We have chosen a longer time-span for solution of the ODE \eqref{eq:Systfini_saturated_FW}, since in this simulation the attraction rate of the closed loop system is much slower than in the previous simulation.
Again, we obtain a considerable approximation of the region of attraction in a computationally efficient way.
\end{remark}

%
%

\vspace{-2mm}
 
\section{Extensions}
\label{sec:extens-obta-results}

\subsection{Pointwise saturations}
\label{sec:Estim_attr_region_infdim_sys_b_times_u}

The type of the saturation which we have considered until now,
i.e. component-wise saturation
of finite-dimensional vectors,  is not the only type of saturation
functions, which appears in engineering practice. A general class of
saturation functions has been considered in \cite{MCP17}.

\subsubsection{Saturation functions}

Various definitions for the saturation map exist.
Normwise saturations limit the norm of the input $u$, i.e. for a given
norm $\|\cdot \|$ on $\R^n$ we may consider
\begin{equation}
\label{sat:norm}
\sat_{\|\cdot \|}(u)\ {:=}
\begin{cases}
u &, \text{ if } \|u\|\leq \ell \\ 
\ell\frac{u}{\| u\|} \hspace{-3mm}&, \text{ if } \| u \|\geq \ell
\end{cases}
 = 
\ell\min\left\{\frac{1}{\ell},\frac{1}{\|u \|}\right\}u.
\end{equation}
Another physically motivated saturation map is the following (pointwise) $L_\infty$  saturation map, defined for all $x \in (0,L)$ by
\begin{equation}
\label{sat:linfinity_saturation}
\begin{split}
\sat_{\infty}(u)(x)
:=&
\begin{cases}
u(x) &, \text{ if } |u(x)|\leq \ell \\ 
\ell\frac{u(x)}{|u(x)|} &, \text{ if } | u(x) |\geq \ell.
\end{cases}\\
 =&
\ell \min\left\{\frac{1}{\ell},\frac{1}{|u(x)|}\right\}u(x).
\end{split}
\end{equation}
In this section, we depart from the saturation model in
(\ref{closed-loop:sat}) and study instead a heat equation with
pointwise saturation in each input channel:
\begin{equation}\label{second:heat:eq}
w_t(t,x){=}w_{xx}(t,x)+c(x)w(t,x)+ \sum_{k=1}^m \sat(b_k(x)u_k(t)).
\end{equation}
In a certain sense, in the equation \eqref{second:heat:eq} the whole terms
$u_kb_k$ are considered as the input which saturates pointwise. The
assumptions on the domain of the problem and the functions $c, b_k$ are
the same as in Section~\ref{sec:Start}. 

In order to stabilize \eqref{second:heat:eq}, we are going to use the same
stabilizing control \eqref{eq:Stabilizing_controller}. We now aim to
provide an estimate for a region of attraction for (\ref{second:heat:eq}). 
So we assume that $\mathbf{K}$ as in \eqref{eq:Stabilizing_controller} is
given and we consider the closed-loop system
\begin{multline*}
\hspace{-3mm}
 w_t(t,x)= w_{xx}(t,x)+c(x)w(t,x)+ 
\sum_{k=1}^m
 \sat\big(b_k(x)\left(\mathbf{K}z(t)\right)_k\big).      
\end{multline*}

Representing this equation in the basis $e_j$, $j\in\N^*$
as in Section~\ref{sec:Start}, we obtain the equations for the coordinates
\begin{equation}
\begin{split}
\dot{w}_j(t)=& \lambda_j w_j(t)  + \sum_{k=1}^m\Big\langle e_j,
\sat_{\infty}\big(b_k(\cdot)\left(\mathbf{K}z(t)\right)_k\big)\Big\rangle
,\ \ j\in\N^*.
\end{split}
\label{eq:Componentwise_equations_with_saturation_1_pointwise_saturations}
\end{equation}

Let us state the following simple result that will be instrumental to bound the differences
\[
\Delta\big( \left(\mathbf{K}z(t)\right)_k, b_k(\cdot)\big):= b_k(\cdot)\sat\big(\mathbf{K}z(t)_k\big)-\sat_{\infty}\big(b_k(\cdot)\mathbf{K}z(t)_k\big).
\]

\begin{lemma}\label{lemma:Delta}
Let $r,k \in \R$ and denote $\Delta(r,k) = r\sat(k)- \sat(rk) $. Then
\begin{equation}
    \label{eq:2}
    |k| \leq \ell \mbox{ and } |rk | \leq \ell \quad \Rightarrow \quad
    \Delta(r,k) = 0.
\end{equation}
\begin{equation}
    \label{eq:3}
    |\Delta(r,k) | \leq \ell ( 1 + |r|).
\end{equation}

Let $k \in \R$ and $b \in L_2(0,L)$ be given. Then
\begin{equation}\label{first:ineq:ell:2}
|k| \leq \ell \mbox{ and } \|b(\cdot) k \|_\infty \leq \ell \quad \Rightarrow \quad \Delta (b(\cdot),k) \equiv 0 \mbox{ a.e. }
\end{equation}
Moreover
\begin{equation}\label{third:ineq:ell}
\|b(\cdot)\sat(k)-\sat_{\infty}(b(\cdot)k)\|_{2}  \leq \ell ( \|1+  |b|\|_2).
\end{equation}
If $|k| \leq \ell$ and $\chi$ is the characteristic function of the set \linebreak
$U:=\{ x \in [ 0, L ]: |kb(x) | > \ell \}$ then 
\begin{equation}\label{fourth:ineq:ell}
\|b(\cdot)\sat(k)-\sat_{\infty}(b(\cdot)k)\|_{2}  \leq \ell ( \| \chi+  |b \chi|\|_2).
\end{equation}
If $b(\cdot)$ is essentially bounded, then
\begin{equation}\label{second:ineq:ell}
\|b(\cdot)\sat(k)-\sat_{\infty}(b(\cdot)k)\|_\infty \leq \ell ( 1+ \| b\|_\infty),
\end{equation}
where $\|b\|_\infty$ denotes the $L_\infty$ norm of the function $b$. 
\end{lemma}

\begin{proof}
The first claim follows as the assumption guarantee that $r\sat(k) =
\sat(rk) = rk$. The claim in \eqref{eq:3} is a direct consequence of the
triangle inequality as
\begin{equation*}
|r\sat(k)- \sat(rk)| \leq |r||\sat(k)| + |\sat(rk)|.
\end{equation*}
The remaining claims follow immediately by applying the pointwise
estimates \eqref{eq:2} and \eqref{eq:3}. The claim \eqref{fourth:ineq:ell}
follows by applying \eqref{third:ineq:ell} to $b\chi$, after noting that
the complement of $U$ does not contribute to the norm of the left hand side.
\end{proof}


\begin{proposition}
    \label{prop:pointwise-sat}
    Consider system \eqref{second:heat:eq} and assume all functions
    $b_k(\cdot) \in L_\infty([ 0, L ])$, $k=1,\ldots,m$.  		
Consider also the
    associated system \eqref{eq:Systfini_saturated_FW} with a stabilizable pair $(\mathbf{A},\mathbf{B}) \in \R^{n\times n}\times\R^{n\times m}$ and let $\mathbf{K}\in\R^{m\times n}$ be such that $\mathbf{A}+\mathbf{B}\mathbf{K}$ is Hurwitz.
Choose a symmetric positive definite matrix
$P\in\mathbb{R}^{n\times n}$ as in Proposition~\ref{prop:Attraction_region_finite_dim_and_infinite_dim_Lyapunov}.
		Consider the Lyapunov function candidate $V$ defined in
    (\ref{def:Lyapunov}).



Then there exist $\alpha, \beta >0$ such that if for $w = z + z^{\perp}
\in \ell_2 (\N^*, \R)$ we have $z\in {\cal A}_\beta := \{z: \; z^\top Pz\leq \beta\}$ then
\begin{equation}
\label{eq:pointwise-lyap-decay}
\dot V(w) \leq -\frac{\alpha}{2} |w|^2.
\end{equation}    
In particular, $\{z: \; z^\top Pz\leq \beta\}\times X_n^\bot$ is a region
of attraction for system \eqref {eq:Componentwise_equations_with_saturation_1_pointwise_saturations}. 
\end{proposition}

\begin{proof}
First we rewrite \eqref{second:heat:eq} by adding and subtracting the term
$\sum b_k(x) \sat(u_k(t))$ to obtain
\begin{multline}
w_t(t,x)=w_{xx}(t,x)+c(x)w(t,x)+ \sum_{k=1}^m b_k(x) \sat(u_k(t)) \\
 + \sum_{k=1}^m\Big(\sat\big(b_k(x)u_k(t)\big) - b_k(x)\sat\big(u_k(t)\big) \Big).
\end{multline}

Define 
\begin{multline}
    \hspace{-4mm}\Delta(t):= \sum_{k=1}^m \sat_{\infty}\big(b_k(\cdot)\left(\mathbf{K}z(t)\right)_k\big) - b_k(\cdot) \sat (\mathbf{K}z(t)_k).
\end{multline}
For $j \in \N^*$ we define
    $y_j(t):=  \langle e_j, \Delta(t) \rangle$, and let  $y(t)\in\R^n$ be the vector with components $y_j(t)$, for $j=1,\ldots,n$.

Considering the Lyapunov function $V$ defined in (\ref{def:Lyapunov}), we compute its time-derivative along the solutions to 
\eqref{eq:Componentwise_equations_with_saturation_1_pointwise_saturations}.
Using (\ref{dot:V:first:PDE}), we get for all $w=z+z^\bot$ with
$z^\top P z \leq 1$ that
\begin{multline}
\dot V(w) \leq -C  |z(t)|^2
- C\| z^\bot(t)\|^2_{\ell_2}+ 2 z(t)^\top P y(t)  
\\
+ 2 \gamma \sum_{j=n+1}^\infty w_j(t) y_j(t).
\label{first:V:second:heat:eq}
\end{multline}


Using \eqref{third:ineq:ell} from Lemma~\ref{lemma:Delta}, we obtain along the solutions to (\ref{second:heat:eq}) and as long as $z^\top Pz \leq 1$,
\begin{align*}
\dot V(w) \leq& -C  |z(t)|^2 - C\| z^\bot(t)\|^2_{\ell_2}+ 2\lambda_{\max} |z(t)| \|\Delta(t)\|_\infty \\
&\qquad\qquad\qquad + 2 \gamma \| z^\bot(t)\|_{\ell_2} \|\Delta(t)\|_\infty,
\end{align*}
where $\lambda _{\max}$ denotes the maximal eigenvalue of the matrix $P$. Therefore for any positive values $\kappa$ and 
$\kappa'$,
\begin{align*}
\dot V(w) \leq& -\big(C- \frac{\lambda_{\max}}{\kappa}\big)  | z(t)|^2
- (C-\frac{\gamma}{\kappa'})\| z^\bot(t)\|^2_{\ell_2}\\
&\qquad\qquad\qquad  +(\lambda_{\max}\kappa   +\gamma \kappa' ) \|\Delta(t)\|_\infty ^2.
\end{align*}
Pick $\kappa>0$ and $\kappa'>0$ such that $C- \frac{\lambda_{\max}}{\kappa}> \frac{3C}{4}$ and $C-\frac{\gamma}{\kappa'}> \frac{C}{2}$. Due to (\ref{first:ineq:ell:2}), there exists $\beta>0$, such that for all $z$ in $\{z: \; z^\top Pz\leq \beta\}$,  $\|\Delta(t)\|_\infty ^2\leq \frac{C}{4(\lambda_{\max}\kappa   +\gamma \kappa' )}| z(t)|^2$. We get, for all solutions to (\ref{second:heat:eq}), as long as $z(t)$ is in $\{z: \; z^\top Pz\leq \beta\}$, 
\begin{equation}\label{lyapunov:second:heat:eq}
\dot V (w) \leq -\frac{C}{2} |z(t)|^2-\frac{C}{2}
\| z^\bot(t)\|^2_{\ell_2}
\end{equation}
Moreover, we have, along the solutions to (\ref{second:heat:eq}),
\begin{equation}
\label{sys-dim-infinie:second:heat:eq}
\dot{w}_j(t)=\lambda_jw_j(t)+\big\langle \sat( b K z(t)),e_j\big\rangle ,\qquad j=1,  2, \ldots
\end{equation}
therefore, considering $V_1$ as previously defined, we may check that
\eqref{eq:pointwise-lyap-decay} holds
following the same computation as for $\dot V$ along the solutions to
(\ref{second:heat:eq}), and using the fact the dynamics
(\ref{sys-dim-infinie:second:heat:eq}) in $X_n$ does not depend on the
component in $X_n^\bot$. 
Therefore the set $\{z: \; z^\top Pz\leq \beta\}\times X_n^\bot$ is invariant along the dynamics to (\ref{second:heat:eq}). With (\ref{lyapunov:second:heat:eq}), we get that $\{z: \; z^\top Pz\leq \beta\}\times X_n^\bot$ is a region of attraction and $V$ is a Lyapunov function.    
\end{proof}

\subsection{Applications to boundary control of heat equation subject to control saturations}
\label{sec:Boundary_control}

%

Let us now start from a heat equation with a dynamical boundary condition
\begin{subequations}
\label{boundary:control:4-sept}
\begin{align}
\hspace{-4mm}y_t(t,x) =&\ y_{xx}(t,x) + c(x) y(t,x), \  t\geq 0, \; x\in (0,L),\\
y(t,0)=&\ 0 ,\; y(t,L)= y_d , \  t\geq 0,
\end{align}
\end{subequations}
where $y_d$ is the (scalar) output of finite-dimensional dynamical system
given by
\begin{subequations}
\label{finite:control:4-sept}
\begin{align}
\dot x_d =& A_d x_ d + B_d \sat  (u(t)), \label{eq:ODE1}\\
y_d =& C_d x_d.\label{eq:ODE2}
\end{align}
\end{subequations}


%
Here $x_d$ in $\R^{n_d}$ is the finite-dimensional state and the dynamics
are subject to a saturating control, $A_d$, $B_d$ and $C_d$ are three
matrices of appropriate dimension, and  $u$ is the scalar control input
for the PDE (\ref{boundary:control:4-sept}) and the ODE
(\ref{finite:control:4-sept}) that is subject to a saturation map. 
Inspired by \cite{PrT19}, we introduce the following change of variable:
\begin{equation*}
w(t,x)= y(t,x)- \frac{x}{L} y_d(t), \quad t\geq 0, \; x\in (0,L).
\end{equation*}
The PDE for $w$ then reads as:
{\allowdisplaybreaks
\begin{eqnarray}
w_t(t,x) &=& y_t(t,x)- \frac{x}{L} \dot{y}_d(t)\nonumber\\
				 &=& y_{xx}(t,x) + c(x)  y(t,x) - \frac{x}{L} C_d \dot{x}_d(t)\nonumber\\
				 &=& w_{xx}(t,x) + c(x) \big(w(t,x) + \frac{x}{L} y_d(t)\big) \nonumber\\
				&&\qquad- \frac{x}{L} C_d \big(A_d x_d(t) + B_d \sat(u(t))\big)\nonumber\\
				 &=& w_{xx}(t,x) + c(x) \big(w(t,x) + \frac{x}{L} C_d x_d(t)\big) \nonumber\\
				&&\qquad- \frac{x}{L} C_d \big(A_d x_d(t) + B_d \sat(u(t))\big)\nonumber\\
				 &=& w_{xx}(t,x) + c(x) w(t,x) \nonumber\\
				&&\qquad+ \big(\underbrace{c(x) \frac{x}{L} C_d - \frac{x}{L} C_d A_d}_{=:d(x)}\big) x_d(t)\nonumber\\
				&&\qquad+\big(\underbrace{- \frac{x}{L} C_d B_d}_{=:b(x)}\big) \sat(u(t)).
\label{boundary:control:4-sept:2}
\end{eqnarray}
}
Please note that $b$ is a scalar function, and $d$ is a row vector function with $d(x) \in \R^{1\times n_d}$, $x\in[0,L]$.

The boundary conditions for the variable $w$ take the form:
\begin{equation}\label{boundary:control:4-sept:2-bc}
w(t,0)=w(t,L)=0 , \; t\geq 0.
\end{equation}
The heat equation \eqref{boundary:control:4-sept:2},
\eqref{boundary:control:4-sept:2-bc} has to be analyzed along with the ODE \eqref{eq:ODE1}.

Performing similar computations as in Section \ref{sec:Start} for the PDE (\ref{closed-loop:sat}) and, using the same notation for $w_j$ and $\lambda_j$, we get
\begin{equation*}
\dot w_ j (t) = \lambda_ j w_ j(t) + b_j \sat (u(t)) + d_ j x_d (t),\ j=1,  2, \ldots,
\end{equation*}
where $b$ and $d$ are defined for $x$ in $[0,L]$ in \eqref{boundary:control:4-sept:2} and $b_j=\langle b(\cdot),e_j(\cdot)\rangle_{L_2(0,L)}$, $d_j=\langle d(\cdot),e_j(\cdot)\rangle_{L_2(0,L)}$, for $j=1,  2, \ldots$. 

Let us consider the first $n$ equations with the ODE (\ref{finite:control:4-sept}) and rewrite this finite-dimensional system as follows:
\begin{eqnarray}
z'(t) &=& \mathbf{A}z(t) + \mathbf{B} \mathbf{K}z(t) + \mathbf{B} \phi(\mathbf{K}z(t))\nonumber\\
&=& \mathbf{A} z(t) + \mathbf{B} \sat(\mathbf{K}z(t)),
\label{eq:Systfini_saturated:4-sept}
\end{eqnarray}
where $\mathbf{K}$ in $\R^{1\times (n+n_d)}$ is a row vector to be designed,
\begin{eqnarray*}
z(t)&:=&(x_d^\top(t),\omega_1(t),\ldots,\omega_n(t))^\top, \quad t\geq 0\\
\mathbf{B} &:=& (B_d^\top,b_1,\ldots,b_n)^\top \in \R^{(n+n_d)\times 1},
\end{eqnarray*}
and 
$\mathbf{A}{:=}
\begin{pmatrix}
A_d & 0\\
D & \Lambda	
\end{pmatrix}
\in \R^{(n+n_d)\times (n+n_d)}$, where
\[
\small
D{:=}
\begin{pmatrix}
d_{11}& d_{12} &\cdots & d_{1n_d}\\
\vdots& \vdots & \vdots & \vdots \\
d_{n1}& d_{n2} &\cdots & d_{nn_d}
\end{pmatrix},
\quad
\Lambda {:=} 
\begin{pmatrix}
\lambda_1 &  & 0 \\
 & \ddots &  \\
0  & & \lambda_n 
\end{pmatrix}
.
\]

As the control input is scalar, we can apply
Proposition~\ref{prop:Stabilization_FiniteDim_Sys_with_Saturation_modified}
to system (\ref{eq:Systfini_saturated:4-sept}) instead of system
\eqref{eq:Systfini_saturated_FW}, and get sufficient conditions for the
estimation of the region of attraction of (\ref{eq:Systfini_saturated:4-sept}). Coming back to the infinite-dimensional systems (\ref{boundary:control:4-sept:2}) and
(\ref{boundary:control:4-sept}), and, applying Proposition~\ref{prop:Attraction_region_finite_dim_and_infinite_dim}, we get sufficient conditions for an estimation of attraction region of (\ref{boundary:control:4-sept}):
\begin{corollary}
Consider system \eqref{eq:Systfini_saturated_FW} with a stabilizable pair $(\mathbf{A},\mathbf{B}) \in \R^{n\times n}\times\R^{n\times 1}$ and let $\mathbf{K}\in\R^{1\times n}$ be such that $\mathbf{A}+\mathbf{B}\mathbf{K}$ is Hurwitz.
Pick a symmetric positive definite {$\tilde{P}\in\mathbb{R}^{(n+n_d)\times (n+n_d)}$} and a $\mathbf{C}\in\mathbb{R}^{1\times (n+n_d)}$ such that \eqref{eq:LMI1_modified} holds.

Then the finite-dimensional system (\ref{eq:Systfini_saturated:4-sept}) is locally exponentially stable in $0$ with a region of attraction given by
\begin{eqnarray}
{\cal A}:=\{z: \; z ^\top \tilde{P} z \leq \mathbf{D}^ {-1} \},
\label{eq:Attraction_Region_D-boundary-control}
\end{eqnarray}
where $\mathbf{D}$ is the minimal real number such that \eqref{eq:LMI2_modified} holds.

Moreover, 
\begin{itemize}
	\item[(i)] \eqref{boundary:control:4-sept:2} is locally exponentially stable with a region of attraction
{$\imath({\cal A}) \times X_n^\perp$,}
	\item[(ii)] \eqref{boundary:control:4-sept} is locally exponentially stable.
\end{itemize}
\end{corollary}

\section{Conclusion}
\label{sec:conclusion}

A linear unstable reaction-diffusion equation has been considered in this paper. Both boundary control and in-domain control cases have been considered. For this control problem, saturated feedback control laws have been designed so that the origin is a locally asymptotically stable equilibrium. The region of attraction has been estimated by an appropriate Lyapunov function and LMI technique. The interest and the efficiency of our approach have been illustrated by means of numerical simulations.

This work leaves several questions open. In particular it could be useful to consider other classes of Lyapunov functions than the ones considered in this work, and to compare the associated estimations of region of attraction. Moreover, it could be interesting to use this work for the estimation of the region of attraction in presence of disturbance and to study local input-to-state stability (as presented in Remark \ref{rem:LISS_wrt_actuator_disturbances}). Finally, let us note that extension of our method to nonlinear systems, other boundary conditions and control input may be considered as a future work.

\vspace{-2mm}

\bibliographystyle{abbrv}
\bibliography{MyPublications,Mir_LitList_NoMir}

\begin{IEEEbiography}[{\includegraphics[width=1in,height=1.25in,clip,keepaspectratio]{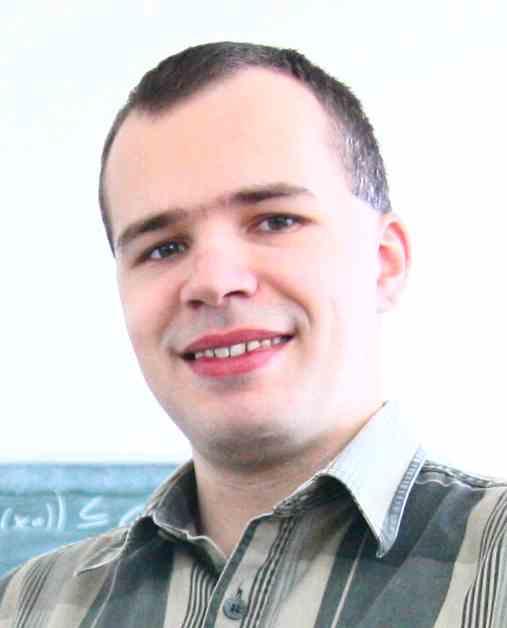}}]{Andrii Mironchenko}
received his MSc at the I.I. Mechnikov Odessa National University in 2008 and his PhD at the University of Bremen in 2012. 
He has held a research position at the University of W\"urzburg and was a Postdoctoral JSPS fellow at the Kyushu Institute of Technology (2013--2014). 
In 2014 he joined the Faculty of Mathematics and Computer Science at the University of Passau. 
His research interests include infinite-dimensional systems, stability theory, hybrid systems and applications of control theory to biological systems. 
\end{IEEEbiography}

\begin{IEEEbiography}[{\includegraphics[width=25mm,height=32mm,clip,keepaspectratio]{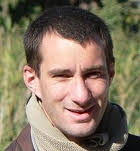}}]{Christophe Prieur}
 was born in 1974. He is currently a senior researcher of the CNRS at the Gipsa-lab, Grenoble, France. He is currently a member of the EUCA-CEB, an associate editor for the  AIMS Evolution Equations and Control Theory and IEEE Trans. on Control Systems Technology, a senior editor for the IEEE Control Systems Letters, and an editor for the IMA Journal of Mathematical Control and Information. He was the Program Chair of the 9th IFAC Symposium on Nonlinear Control Systems (NOLCOS 2013) and of the 14th European Control Conference (ECC 2015). His current research interests include nonlinear control theory, hybrid systems, and control of partial differential equations. 
\end{IEEEbiography} 

\begin{IEEEbiography}[{\includegraphics[width=1in,height=1.25in,clip,keepaspectratio]{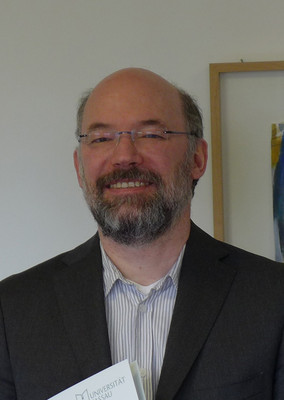}}]{Fabian Wirth}
received his PhD from the Institute of Dynamical Systems at
the University of Bremen in 1995. He has since held positions at the Centre Automatique et Syst{\`e}mes of Ecole des Mines, the
Hamilton Institute at NUI Maynooth, Ireland, the University of W\"urzburg and IBM Research Ireland. 
He now holds the chair for Dynamical Systems at the University
of Passau. His current interests include stability
theory, switched systems and large scale networks with applications to networked systems and in the domain of smart cities.
\end{IEEEbiography}

\appendix{} 

\subsection{Series expansion of solutions}
\label{appendix2}

Denoting the saturated nonlinearity
in \eqref{closed-loop:sat} by $f:\R_+ \to X$, we have for the mild solution of \eqref{newzero} that
\begin{equation*}
w(t)=T(t)w(0)+\int_{0}^{t}T(t-s)f(s)ds\,,
\end{equation*}
where $T(t)$ is the strongly continuous semigroup generated by $A$.
Here the integral on the right-hand side is well-defined as the Bochner integral,
see \cite[Example A.1.13]{JaZ12}.

Let $\left(e_j\right)_{j\geq 1}$ be the Hilbert basis of $X$ given by the eigenfunctions of $A$. Then we can define
\begin{align*}
w_j(t) &:=\left\langle w(t), e_j\right\rangle\\
       &=\left\langle T(t) w(0), e_j\right\rangle + \left\langle\int_{0}^{t}T(t-s)f(s)ds, e_j\right\rangle
\end{align*}
and by \cite[Corollary V.5.2]{Yos80} we may interchange the linear map
$\langle \cdot, e_j \rangle$ with the integral to obtain
\begin{align*}
w_j(t) &= \left\langle T(t)w(0),e_j\right\rangle +\int_{0}^{t}\left\langle T(t-s)f(s), e_j\right\rangle ds\\
 &=\left\langle w(0), T(t)^{\ast}e_j\right\rangle + \int_{0}^{t}\left\langle f(s), T(t-s)^{\ast}e_j\right\rangle ds\\
 &=e^{\lambda_jt}\left\langle w(0), e_j\right\rangle + \int_{0}^{t}e^{\lambda_j(t-s)}\left\langle f(s), e_j\right\rangle ds\,.
\end{align*}
Here the integral on the right is a standard Lebesque
integral and so $w_j$ solves an integral equation. Thus $w_j$ is absolutely
continuous and satisfies, for almost all $t$, the Carath{\'e}odory equation
\begin{align*}
\dot{w}_j(t) &=\lambda_jw_j(t)+\left\langle \sum_{k=1}^{m}b_k\mathrm{sat}
  \left(u_k(t)\right), e_j\right\rangle \\ 
&=\lambda_j w_j(t)+\sum_{k=1}^{m}b_{kj}\mathrm{sat} \left(u_k(t)\right)\,.
\end{align*}
This justifies the consideration of \eqref{sys-dim-infinie} as an
equivalent system for \eqref{newzero}.

\subsection{Compactness of the resolvent for Sturm-Liouville operators}
\label{appendix}

The following discussion summarizes some results from \cite{SaS03}, which
provide the necessary arguments to show the compactness of the resolvent
of the operator
$A$ introduced in Section~\ref{sec:Start}.

Here we use the following notation. All function spaces are considered on
the interval $[0,L]$. The Sobolev space $W^{k}_p$ is the
space of $L_p$-functions ($1\leq p<\infty$) such that the function is
$k$-times weakly differentiable and the corresponding derivatives are
again in $L_p$. Note in particular that $W^{k}_2=H^k$. For negative
indices, we set $W^{-1}_2 := \left(W^{1}_{0,2}\right)^*$ (the dual to $W^{1}_{0,2}=H^1_0$).


Recall that an operator $A \in L(X)$ is said to be compact, if $A$
maps bounded sets into precompact sets. For a densely defined linear
operator $(A,D(A)):X\to X$ with a nonempty resolvent set $\rho(A)$, it is an easy consequence of the
resolvent identity that  the resolvent
$R_\lambda(A)$ is compact for some $\lambda$ in the resolvent set, if and
only if it is compact on the entire resolvent set,
see \cite[Theorem~III.6.29]{Kat95}.



\begin{definition}
\label{def:Operator-with-a-compact-resolvent} 
We say that a closed densely defined linear operator $(A,D(A)):X\to X$ has
a compact resolvent, if there exists a $\lambda \in\rho(A)$ so that $R_\lambda(A)$ is compact.
\end{definition}

\subsubsection{Some results from \cite{SaS03}}

We note that in \cite{SaS03} the case $L=\pi$ is
considered, which requires some rescaling to use their results.

Let $X:=L_2(0,L)$, $q \in W^{-1}_2(0,L)$, and define (for $y \in W^{1}_1(0,L)$) the quasiderivative
\begin{eqnarray}
y^{[1]}(x):= \frac{dy}{dx} - Q(x)y(x),
\label{eq:Quasiderivative}
\end{eqnarray}
where 
\[
Q(x):= \int_0^xq(s)ds.
\]

For $q \in W^{-1}_2(0,L)$ it holds that $Q \in X$, and for $q \in X$ it holds that $Q\in W^1_2(0,L)$.

Let $q\in L_1(0,L)$ be given.
Consider the formal Sturm-Liouville operator $SL:X\to X$ defined by 
\[
SL(y):= - \frac{d^2y}{dx^2} + q(x)y(x), \quad x \in (0,L), 
\]
where $(0,L) \subset \R$.
In order to fully define the operator $SL$, we have to introduce its domain of definition.
Following \cite[Section 1.1]{SaS03} we define the maximal operator $L_M$, defined by
\begin{subequations}
\label{eq:L_M-operator}
\begin{align}
L_My =& SL y, \\
D(L_M):=&\{y : y,y^{[1]} \in W^{1}_1(0,L), SL(y) \in X\}.
\end{align}
\end{subequations}

The following result has been shown in \cite[Theorem 1.5]{SaS03}:
\begin{theorem}
\label{thm:Theorem-1.5-SaS} 
Let the operator $A$ be the restriction of $L_M$ to the domain
\begin{eqnarray}
D(A):=\{y\in D(L_M): U_1(y)=U_2(y) = 0\},
\label{eq:SaS-theorem}
\end{eqnarray}
where for $j=1,2$ it holds that
\begin{eqnarray}
U_j(y) = a_{j1} y(0) + a_{j2} y^{[1]}(0) + b_{j1} y(L) + a_{j2} y^{[1]}(L),
\label{eq:SL-boundary-conditions}
\end{eqnarray}
where $a_{j1}, a_{j2}, b_{j1} , a_{j2}$ are real numbers, for $j=1,2$.
 
Let $J_{\alpha\beta}$ be the determinant of the $\alpha$-th and $\beta$-th column of a matrix 
\begin{eqnarray}
\begin{pmatrix}
a_{11} & a_{12} & b_{11} & b_{12}\\	
a_{21} & a_{22} & b_{21} & b_{22}
\end{pmatrix}.
\label{eq:BC-coefs-matrix}
\end{eqnarray}
Then the operator $A$ has a nonempty resolvent set $\rho(A)$, has a compact resolvent and discrete spectrum, if one of the following conditions holds:
\begin{enumerate}
	\item[(i)] $J_{42}\neq 0$,
	\item[(ii)] $J_{42}= 0$, $J_{14}+J_{32}\neq 0$,
	\item[(iii)] $J_{42}= J_{14} = J_{32} = 0$, $J_{12}+J_{34} =  0$, $J_{13} \neq 0$.
\end{enumerate} 
\end{theorem}

\begin{remark}
\label{rem:Compactness of a resolvent} 
Compactness of the resolvent of $A$ is not mentioned in the formulation of \cite[Theorem 1.5]{SaS03}, but 
its compactness was shown in the proof.
\end{remark}

\begin{remark}
\label{rem:Birkhoff-regularity.} 
If one of conditions (i)--(iii) holds, then the boundary conditions \eqref{eq:SL-boundary-conditions} are called
\emph{Birkhoff-regular}.
\end{remark}

\subsubsection{Application to the system in Section~\ref{sec:Start}}

Consider the operator $A$ as in Section~\ref{sec:Start}, i.e.
\begin{subequations}
\label{eq:Our-operator-A}
\begin{eqnarray}
A&:=&\partial_{xx}+c(\cdot)\mathrm{id}: X \to X,\\
D(A)&=&H^2(0,L)\cap H^1_0(0,L)
\end{eqnarray}
\end{subequations}

Let us show that Theorem~\ref{thm:Theorem-1.5-SaS} can be used for our operator $A$.
We need the following lemma, see \cite[Exercise 4 at p. 306]{Eva10}.
\begin{lemma}
\label{lem:Characterization-of-Sobolev-spaces-in-1D} 
Let $p\in [1,+\infty)$. Then $f \in W^{1}_p(0,L)$ if and only if $f$ is equal a.e. to an absolutely continuous function,
the derivative $f'$ exists a.e. and is an element of $L_p(0,L)$.
\end{lemma}

First we show
\begin{lemma}
\label{lem:Restatement of the domain} 
The domain $D(A)$ of the operator \eqref{eq:Our-operator-A} has the form \eqref{eq:SaS-theorem} for $U_1(y) = y(0)$ and $U_2(y)=y(L)$.
\end{lemma}

\begin{proof}
Let $AC((0,L),\R)$ be the space of absolutely continuous functions from $(0,L)$ to $\R$.
In view of Lemma~\ref{lem:Characterization-of-Sobolev-spaces-in-1D} 
the set $D(A)$ can be rewritten as
\begin{eqnarray}
D(A) &=& \{f\in AC((0,L),\R):\frac{df}{dx} \in AC((0,L),\R), \nonumber\\
 &&\qquad \frac{d^2f}{dx^2} \in L_2(0,L),\ f(0)=f(L)=0\}.
\label{eq:Domain-restatement-1}
\end{eqnarray}

Now consider the domain defined in  \eqref{eq:SaS-theorem}.

As we restrict our attention to the case when $q \in X$, for any $y \in W^1_1 (0,L)$ it holds that 
$Q(\cdot)y(\cdot) \in W^1_1 (0,L)$, and thus $y^{[1]} \in W^1_1 (0,L)$ if and only if $\frac{dy}{dx} \in W^1_1 (0,L)$.

Hence the domain $D(L_M)$ can be equivalently written as 
\[
D(L_M)=\{f : f, \frac{df}{dx} \in W^{1}_1(0,L), SL(f) \in X\}.
\]
Using again Lemma~\ref{lem:Characterization-of-Sobolev-spaces-in-1D}, we see that
\begin{eqnarray}
D(L_M)&=&\{f\in AC((0,L),\R): \frac{df}{dx} \in AC((0,L),\R), \nonumber \\
&&\qquad \frac{d^2f}{dx^2} \in L_1(0,L),\  SL(f) \in X\}.
\end{eqnarray}
Furthermore, as $f \in D(L_M)$ is absolutely continuous on $(0,L)$, and  $q \in X$, it holds that $qf \in X$ and it holds that $SL(f) \in X$ if and only if $\frac{d^2f}{dx^2} \in X$.
As $X \subset L_1(0,L)$, we can finally restate $D(L_M)$ as
\begin{eqnarray}
D(L_M)&=&\{f\in AC((0,L),\R): \frac{df}{dx} \in AC((0,L),\R),\nonumber \\
&&\qquad \frac{d^2f}{dx^2} \in L_2(0,L)\}.
\end{eqnarray}
Using the boundary conditions $U_1(y) = y(0)$ and $U_2(y)=y(L)$, we see that 
the set \eqref{eq:SaS-theorem} is precisely \eqref{eq:Domain-restatement-1}.
\end{proof}

In view of Lemma~\ref{lem:Restatement of the domain}, we can use Theorem~\ref{thm:Theorem-1.5-SaS} 
to study the spectral properties of the Sturm-Liouville operator $A$.

\begin{proposition}
\label{prop:Compact-resolvent-of-A} 
The operator \eqref{eq:Our-operator-A} has a compact resolvent.
\end{proposition}

\begin{proof}
For the boundary conditions $U_1(y) = y(0)$ and $U_2(y)=y(L)$
the coefficients $a_{ij}$ and $b_{ij}$ from the formulation of Theorem~\ref{thm:Theorem-1.5-SaS}
have the form:
$a_{11} = 1$, $b_{21} = 1$, and all other entries of a matrix in \eqref{eq:BC-coefs-matrix} are zeros.
By item (iii) of Theorem~\ref{thm:Theorem-1.5-SaS} the boundary conditions are Birkhoff-regular and the operator $A$ has a compact resolvent.
\end{proof}

\end{document}